\documentclass{article}%
\usepackage{amsmath}
\usepackage{amsfonts}
\usepackage{amssymb}
\usepackage{graphicx}%
\providecommand{\U}[1]{\protect\rule{.1in}{.1in}}
\newtheorem{theorem}{Theorem}

\newtheorem{corollary}[theorem]{Corollary}

\newtheorem{definition}[theorem]{Definition}
\newtheorem{example}[theorem]{Example}

\newtheorem{lemma}[theorem]{Lemma}

\newtheorem{problem}[theorem]{Problem}
\newtheorem{proposition}[theorem]{Proposition}

\newenvironment{proof}[1][Proof]{\noindent\textbf{#1.} }{\ \rule{0.5em}{0.5em}}

\def\U{{\mathbb U}}

\def\1{\textcircled{1}}
\def\2{\textcircled{2}}

\begin{document}

\title{On PBZ$^{\ast}$ --lattices}
\author{Roberto Giuntini$^{1}$, Claudia Mure\c{s}an$^{2}$, Francesco Paoli$^{1}$
\and $^{1}$Dept. of Pedagogy, Psychology, Philosophy, University of Cagliari
\and $^{2}$Faculty of Mathematics and Computer Science,
\and University of Bucharest}
\maketitle

\begin{abstract}
We continue our investigation of paraorthomodular BZ*-lattices (PBZ$^{\ast}$
--lattices), started in\ \cite{GLP1+, PBZ2, rgcmfp, pbzsums, pbz5}. We shed
further light on the structure of the subvariety lattice of the variety
$\mathbb{PBZL}^{\mathbb{\ast}}$ of PBZ$^{\ast}$ --lattices; in particular, we
provide axiomatic bases for some of its members. Further, we show that some
distributive subvarieties of $\mathbb{PBZL}^{\mathbb{\ast}}$ are
term-equivalent to well-known varieties of expanded Kleene lattices or of
nonclassical modal algebras. By so doing, we somehow help the reader to locate
PBZ$^{\ast}$ --lattices on the atlas of algebraic structures for nonclassical logics.

\end{abstract}

\section{Introduction}

One of the core topics within the impressive \emph{corpus} of Mohammad
Ardeshir's contributions to mathematical logic is the algebraic semantics of
nonclassical logics. In particular, Ardeshir and his collaborators intensively
investigated the relationships between Visser's basic propositional calculus
\cite{Viss} and its algebraic counterpart, \emph{basic algebras},
generalisations of Heyting algebras where only the left-to-right direction of
the residuation equivalence $x\wedge y\leq z\Longleftrightarrow x\leq
y\rightarrow z$ is retained \cite{Ard1, Ard2, Ard3}. Also, in a basic algebra
$\mathbf{A}$ there may be $a\in A$ such that $1\rightarrow a\neq a$.
Crucially, the introduction of these structures is not motivated by
abstraction \emph{per se}: Ardeshir argues that basic algebras can contribute
to a deeper understanding of constructive mathematics, whence they can have a
paramount\emph{ foundational} interest.

The approach that led to the introduction of paraorthomodular BZ*-lattices
(\emph{PBZ}$^{\ast}$\emph{ --lattices}) \cite{GLP1+, PBZ2, rgcmfp, pbzsums,
pbz5} is similar. The key motivation for this particular generalisation of
orthomodular lattices, in fact, comes from the foundations of quantum
mechanics. Consider the structure%
\[
\mathbf{E}\left(  \mathbf{H}\right)  =\left(  \mathcal{E}\left(
\mathbf{H}\right)  ,\wedge_{s},\vee_{s},^{\prime},^{\sim},\mathbb{O}%
,\mathbb{I}\right)  \text{,}%
\]
where:

\begin{itemize}
\item $\mathcal{E}\left(  \mathbf{H}\right)  $ is the set of all effects of a
given complex separable Hilbert space $\mathbf{H}$, i.e., positive linear
operators of $\mathbf{H}$ that are bounded by the identity operator
$\mathbb{I}$;

\item $\wedge_{s}$ and $\vee_{s}$ are the meet and the join, respectively, of
the \emph{spectral ordering} $\leq_{s}$ so defined for all $E,F\in
\mathcal{E}\left(  \mathbf{H}\right)  $:
\[
E\leq_{s}F\,\,\,\text{iff}\,\,\,\forall\lambda\in\mathbb{R}:\,\,M^{F}%
(\lambda)\leq M^{E}(\lambda),
\]
where for any effect $E$, $M^{E}$ is the unique spectral family \cite[Ch.
7]{Kr} such that $E=\int_{-\infty}^{\infty}\lambda\,dM^{E}(\lambda)$ (the
integral is here meant in the sense of norm-converging Riemann-Stieltjes sums
\cite[Ch. 1]{Strocco});

\item $\mathbb{O}$ and $\mathbb{I}$ are the null and identity operators, respectively;

\item $E^{\prime}=\mathbb{I}-E$ and $E^{\sim}=P_{\ker\left(  E\right)  }$ (the
projection onto the kernel of $E$).
\end{itemize}

The operations in $\mathbf{E}\left(  \mathbf{H}\right)  $ are well-defined.
The spectral ordering is indeed a lattice ordering \cite{Ols, deG} that
coincides with the usual ordering of effects induced via the trace functional
when both orderings are restricted to the set of projection operators of the
same Hilbert space.

A {PBZ$^{\ast}$ --lattice can be viewed as an abstraction from this }concrete
physical model, much in the same way as an orthomodular lattice can be viewed
as an abstraction from a certain structure of projection operators in a
complex separable Hilbert space. The faithfulness of {PBZ$^{\ast}$ --lattice}s
to the physical model whence they stem is further underscored by the fact that
they reproduce at an abstract level the \textquotedblleft
collapse\textquotedblright\ of several notions of \emph{sharp physical
property} that can be observed in $\mathbf{E}\left(  \mathbf{H}\right)  $.

Referring the reader to \cite{GLP1+} for a more detailed discussion of the
previous issues, we now summarise the discourse of the present paper. In
Section \ref{bari} we collect some preliminaries, with the twofold aim of
fixing the notation to be used throughout the article and of making the
article itself sufficiently self-contained --- although we will occasionally
need to refer the reader to results included in the previous papers on the
subject. In Section \ref{lecce} we zoom in on some subvarieties of the variety
$\mathbb{PBZL}^{\mathbb{\ast}}$ of PBZ$^{\ast}$ --lattices. First, we
axiomatise the subvariety of $\mathbb{PBZL}^{\mathbb{\ast}}$ generated by a
particular algebra whose role in the context of $\mathbb{PBZL}^{\mathbb{\ast}%
}$ is analogous to the role of the benzene ring in the context of
ortholattices. Next, we prove that the subvariety of $\mathbb{PBZL}%
^{\mathbb{\ast}}$ generated by the (unique PBZ$^{\ast}$ --lattice over the)
$4$-element Kleene chain is the unique antiorthomodular cover of the variety
generated by the (unique PBZ$^{\ast}$ --lattice over the) $3$-element Kleene
chain. Finally, we put to good use the construction of subdirect products of
varieties of PBZ$^{\ast}$ --lattices, employing them to characterise some
joins of subvarieties of PBZ$^{\ast}$ --lattices. Section \ref{thedis} is
devoted to term-equivalence results that establish connections between
\emph{distributive} varieties of PBZ$^{\ast}$ --lattices and some known
expansions of Kleene lattices, on the one hand, and nonclassical modal
algebras --- i.e., modal algebras whose nonmodal reducts are generic De Morgan
algebras rather than Boolean algebras --- on the other. We hope that these
equivalences can help readers to make out the whereabouts of PBZ$^{\ast}$
--lattices in the vast landscape of algebraic structures for nonclassical
logic, a territory whose exploration has been decisively aided by the research
work of Mohammad Ardeshir.

\section{Preliminaries\label{bari}}

For further information on the notions recalled in this section, we refer the
reader to \cite{GLP1+,PBZ2,rgcmfp,pbzsums,pbz5}.

We denote by ${\mathbb{N}}$ the set of the natural numbers and by
${\mathbb{N}}^{\ast}={\mathbb{N}}\setminus\{0\}$. If $\mathbf{A}$ is an
algebra, then $A$ will denote its universe. We call \emph{trivial algebras}
the singleton algebras. For any $n\in{\mathbb{N}}^{\ast}$, $\mathbf{D}_{n}$
will denote the $n$--element chain, as well as any bounded lattice-ordered
structure having this chain as a bounded lattice reduct. For any lattice
$\mathbf{L}$, we denote by $\mathbf{L}^{d}$ the dual of $\mathbf{L}$. For any
bounded lattices $\mathbf{L}$ and $\mathbf{M}$, we denote by $\mathbf{L}%
\oplus\mathbf{M}$ the ordinal sum of $\mathbf{L}$ with $\mathbf{M}$, obtained
by glueing together the top element of $\mathbf{L}$ and the bottom element of
$\mathbf{M}$, thus stacking $\mathbf{M}$ on top of $\mathbf{L}$, and by
$L\oplus M$ the universe of the bounded lattice $\mathbf{L}\oplus\mathbf{M}$;
clearly, the ordinal sum of bounded lattices is associative.

Let ${\mathbb{V}}$ be a variety of algebras of similarity type $\tau$ and
${\mathbb{C}}$ a class of algebras with $\tau$--reducts. We denote by
${\mathrm{I}}_{{\mathbb{V}}}({\mathbb{C}})$, ${\mathrm{H}}_{{\mathbb{V}}%
}({\mathbb{C}})$, ${\mathrm{S}}_{{\mathbb{V}}}({\mathbb{C}})$ and
${\mathrm{P}}_{{\mathbb{V}}}({\mathbb{C}})$ the classes of the isomorphic
images, homomorphic images, subalgebras and direct products of $\tau$--reducts
of members of ${\mathbb{C}}$, respectively, and by $V_{{\mathbb{V}}%
}({\mathbb{C}})={\mathrm{H}}_{{\mathbb{V}}}{\mathrm{S}}_{{\mathbb{V}}%
}{\mathrm{P}}_{{\mathbb{V}}}({\mathbb{C}})$ the subvariety of ${\mathbb{V}}$
generated by the $\tau$--reducts of the members of ${\mathbb{C}}$. For any
class operator $\mathrm{O}$ and any $\mathbf{A}\in{\mathbb{C}}$, the notation
$\mathrm{O}_{{\mathbb{V}}}(\{\mathbf{A}\})$ will be streamlined to
$\mathrm{O}_{{\mathbb{V}}}(\mathbf{A})$. If $\mathbf{A}$ is an algebra having
a $\tau$--reduct, $n\in{\mathbb{N}}$ and $\kappa_{1},\ldots,\kappa_{n}$ are
constants over $\tau$, then we denote by $\mathrm{Con}_{{\mathbb{V}}%
}(\mathbf{A})$ the complete lattice of the congruences of the $\tau$--reduct
of $\mathbf{A}$, as well as the set reduct of this congruence lattice, and by
$\mathrm{Con}_{{\mathbb{V}}\kappa_{1},\ldots,\kappa_{n}}(\mathbf{A})$ the
complete sublattice of $\mathrm{Con}_{{\mathbb{V}}}(\mathbf{A})$ consisting of
the congruences with singleton classes of $\kappa_{1}^{\mathbf{A}}%
,\ldots,\kappa_{n}^{\mathbf{A}}$, as well as its set reduct. If ${\mathbb{V}}$
is the variety of lattices or that of bounded lattices, then the subscript
$_{{\mathbb{V}}}$ will be eliminated from the previous notations. If
${\mathbb{C}}\subseteq{\mathbb{V}}$, then we denote by $Si({\mathbb{C}})$ the
class of the members of ${\mathbb{C}}$ which are subdirectly irreducible in
${\mathbb{V}}$. The lattice of subvarieties of ${\mathbb{V}}$ and its set
reduct will be denoted by $\mathrm{Subvar}({\mathbb{V}})$.

An \emph{involution lattice} (in brief, \emph{I--lattice}) is an algebra
$\mathbf{L}=(L,\wedge,\vee,\cdot^{\prime})$ of type $(2,2,1)$ such that
$(L,\wedge,\vee)$ is a lattice and $\cdot^{\prime}:L\rightarrow L$ is an
order--reversing operation that satisfies $a^{\prime\prime}=a$ for all $a\in
L$. This makes $\cdot^{\prime}$ a dual lattice automorphism of $\mathbf{L}$,
called \emph{involution}.

A \emph{bounded involution lattice} (in brief, \emph{BI--lattice}) is an
algebra $\mathbf{L}=(L,\wedge,\vee,$\linebreak$\cdot^{\prime},0,1)$ of type
$(2,2,1,0,0)$ such that $(L,\wedge,\vee,0,1)$ is a bounded lattice and
$(L,\wedge,\vee,\cdot^{\prime})$ is an involution lattice. A distributive
bounded involution lattice is called a \emph{De Morgan algebra}.

For any BI--lattice $\mathbf{L}$, we denote by $S(\mathbf{L})$ the set of the
\emph{sharp elements} of $\mathbf{L}$, that is: $S(\mathbf{L})=\{x\in L:x\vee
x^{\prime}=1\}$. A BI--lattice $\mathbf{L}$ is called an \emph{ortholattice}
iff all its elements are sharp, and it is called an \emph{orthomodular
lattice} iff, for all $a,b\in L$, $a\leq b$ implies $b=(b\wedge a^{\prime
})\vee a$.

A \emph{pseudo--Kleene algebra} is a BI--lattice $\mathbf{L}$ that satisfies
$a\wedge a^{\prime}\leq b\vee b^{\prime}$ for all $a,b\in L$. The involution
of a pseudo--Kleene algebra is called \emph{Kleene complement}. Distributive
pseudo--Kleene algebras are called \emph{Kleene algebras} or \emph{Kleene
lattices}.

Clearly, for any bounded lattice $\mathbf{L}$ and any BI--lattice $\mathbf{K}%
$, if $\mathbf{K}_{l}$ is the bounded lattice reduct of $\mathbf{K}$, then the
bounded lattice $\mathbf{L}\oplus\mathbf{K}_{l}\oplus\mathbf{L}^{d}$ becomes a
BI--lattice with the involution that restricts to the involution of
$\mathbf{K}$ on $K$, to a dual lattice isomorphism from $\mathbf{L}$ to
$\mathbf{L}^{d}$ on $L$ and to the inverse of this lattice isomorphism on
$L^{d}$. This BI--lattice, which we denote by $\mathbf{L}\oplus\mathbf{K}%
\oplus\mathbf{L}^{d}$, is a pseudo--Kleene algebra iff $\mathbf{K}$ is a
pseudo--Kleene algebra.

We denote by $\mathbb{BA} $, $\mathbb{OML} $, $\mathbb{OL} $, $\mathbb{KA} $,
$\mathbb{PKA} $, $\mathbb{BI} $ and ${\mathbb{I}} $ the varieties of Boolean
algebras, orthomodular lattices, ortholattices, Kleene algebras,
pseudo--Kleene algebras, BI--lattices and I--lattices, respectively. Note that
$\mathbb{BA} \subsetneq\mathbb{OML} \subsetneq\mathbb{OL} \subsetneq
\mathbb{PKA} \subsetneq\mathbb{BI} $ and $\mathbb{BA} \subsetneq\mathbb{KA}
\subsetneq\mathbb{PKA} $.

An algebra $\mathbf{A}$ having a BI--lattice reduct is said to be
\emph{paraorthomodular} iff, for all $a,b\in A$, if $a\leq b$ and $a^{\prime
}\wedge b=0$, then $a=b$. Note that orthomodular lattices are paraorthomodular
and that paraorthomodular ortholattices are orthomodular lattices.

A \emph{Brouwer--Zadeh lattice} (in brief, \emph{BZ--lattice}) is an algebra
$\mathbf{L}=(L,\wedge,\vee,\cdot^{\prime},$\linebreak$\cdot^{\sim},0,1)$ of
type $(2,2,1,1,0,0)$ such that $(L,\wedge,\vee,\cdot^{\prime},0,1)$ is a
pseudo--Kleene algebra and $\cdot^{\sim}:L\rightarrow L$ is an
order--reversing operation, called \emph{Brouwer complement}, that satisfies:
$a\wedge a^{\sim}=0$ and $a\leq a^{\sim\sim}=a^{\sim\prime}$ for all $a\in L$.
In any BZ--lattice $\mathbf{L}$, we denote by $\square a=a^{\prime\sim}$ and
by $\Diamond a=a^{\sim\sim}$ for all $a\in L$. Note that, in any BZ--lattice
$\mathbf{L}$, we have, for all $a,b\in L$: $a^{\sim\sim\sim}=a^{\sim}\leq
a^{\prime}$, $(a\vee b)^{\sim}=a^{\sim}\wedge b^{\sim}$ and $(a\wedge
b)^{\sim}\geq a^{\sim}\vee b^{\sim}$. The class of BZ-lattices is a variety,
hereafter denoted by $\mathbb{BZL}$.

We consider the following equations over $\mathbb{BZL}$, out of which
SDM\ (the \emph{Strong De Morgan} identity) clearly implies $(\ast)$, as well
as \mbox{\textup{SK}}, while \mbox{\textup{J0}}\ implies J2:

\begin{flushleft}%
\begin{tabular}
[c]{rl}%
$(\ast)$ & $(x\wedge x^{\prime})^{\sim}\approx x^{\sim}\vee x^{\prime\sim}$\\
\mbox{\textup{SDM}} & $(x\wedge y)^{\sim}\approx x^{\sim}\vee y^{\sim}$\\
\mbox{\textup{SK}} & $x\wedge\Diamond y\leq\square x\vee y$\\
\mbox{\textup{DIST}} & $x\wedge(y\vee z)\approx(x\wedge y)\vee(x\wedge z)$\\
\mbox{\textup{J0}} & $(x\wedge y^{\sim})\vee(x\wedge\Diamond y)\approx x$\\
$\textup{J}2$ & $(x\wedge(y\wedge y^{\prime})^{\sim})\vee(x\wedge
\Diamond(y\wedge y^{\prime}))\approx x$%
\end{tabular}

\end{flushleft}

A \emph{PBZ$^{\ast}$ --lattice} is a paraorthomodular BZ--lattice that
satisfies equation $(\ast)$. In any PBZ$^{\ast}$ --lattice $\mathbf{L}$,%
\[
S(\mathbf{L})=\{a^{\sim}:a\in L\}=\{a\in L:a^{\sim\sim}=a\}=\{a\in
L:a^{\prime}=a^{\sim}\}
\]
and $S(\mathbf{L})$ is the universe of the largest orthomodular subalgebra of
$\mathbf{L}$, that we denote by $\mathbf{S}(\mathbf{L})$.

We denote by $\mathbb{PBZL}^{\ast}$ the variety of PBZ$^{\ast}$ --lattices;
note that paraorthomodularity becomes an equational condition under the
$\mathbb{BZL}$ axioms and condition $(\ast)$. We also denote by $\mathbb{DIST}%
=\{\mathbf{L}\in\mathbb{PBZL}^{\ast}:\mathbf{L}\vDash\mbox{\textup{DIST}}\}$.
By the above, $\mathbb{OML}$ can be identified with the subvariety
$\{\mathbf{L}\in\mathbb{PBZL}^{\ast}:\mathbf{L}\vDash x^{\prime}\approx
x^{\sim}\}$ of $\mathbb{PBZL}^{\ast}$, by endowing each orthomodular lattice,
in particular every Boolean algebra, with a Brouwer complement equalling its
Kleene complement. With the same extended signature, $\mathbb{OL}$ becomes the
subvariety $\{\mathbf{L}\in\mathbb{BZL}:\mathbf{L}\vDash x^{\prime}\approx
x^{\sim}\}$ of $\mathbb{BZL}$.

A PBZ$^{\ast}$ --lattice $\mathbf{A}$ with no nontrivial sharp elements, that
is with $S(\mathbf{A})=\{0,1\}$, is called an \emph{antiortholattice}. A
PBZ$^{\ast}$ --lattice $\mathbf{A}$ is an antiortholattice iff it is endowed
with the following Brouwer complement, called the \emph{trivial Brouwer
complement}: $0^{\sim}=1$ and $a^{\sim}=0$ for all $a\in A\setminus\{0\}$.
Every paraorthomodular pseudo--Kleene algebra with no nontrivial sharp
elements becomes an antiortholattice when endowed with the trivial Brouwer
complement. In particular, any BZ--lattice with the $0$ meet--irreducible, and
thus any BZ--chain, is an antiortholattice. Moreover, BZ--lattices with the
$0$ meet--irreducible are exactly the antiortholattices that satisfy
\mbox{\textup{SDM}}. Also, if $\mathbf{L}$ is a nontrivial bounded lattice and
$\mathbf{K}$ is a pseudo--Kleene algebra, then the pseudo--Kleene algebra
$\mathbf{L}\oplus\mathbf{K}\oplus\mathbf{L}^{d}$, endowed with the trivial
Brouwer complement, becomes an antiortholattice, that we will also denote by
$\mathbf{L}\oplus\mathbf{K}\oplus\mathbf{L}^{d}$.

Antiortholattices form a proper universal class, denoted by $\mathbb{AOL}$.
Clearly, $\mathbb{AOL}\cup\mathbb{OML}\subsetneq\mathbb{PBZL}^{\ast}%
\subsetneq\mathbb{BZL}\supsetneq\mathbb{OL}$. Note, also, that $\mathbb{OML}%
\cap V_{\mathbb{BZL}}(\mathbb{AOL})=\mathbb{OML}\cap\mathbb{DIST}=\mathbb{BA}%
$, hence $\mathbb{DIST}\subsetneq V_{\mathbb{BZL}}(\mathbb{AOL})$. We denote
by $\mathbb{SDM}=\{\mathbf{L}\in\mathbb{PBZL}^{\ast}:\mathbf{L}\vDash
\mbox{\textup{SDM}}\}$ and by $\mathbb{SAOL}=\mathbb{SDM}\cap V_{\mathbb{BZL}%
}(\mathbb{AOL})$.

If $\mathbf{L}$ is a nontrivial bounded lattice and ${\mathbb{C}}$ is a class
of bounded lattices, BI--lattices or pseudo--Kleene algebras, then we denote
by $\mathbf{L}\oplus{\mathbb{C}}\oplus\mathbf{L}^{d}$ the following class of
bounded lattices, BI--lattices or antiortholattices:%
\[
\mathbf{L}\oplus{\mathbb{C}}\oplus\mathbf{L}^{d}=\{\mathbf{L}\oplus
\mathbf{A}\oplus\mathbf{L}^{d}:\mathbf{A}\in{\mathbb{C}}\}.
\]

\section{A Study of Some Subvarieties\label{lecce}}

Throughout this section, the results cited from \cite{pbz5} will be numbered
as in the third arXived version of this paper.

\subsection{The $\mathbf{F}_{8}$ Problem}

There is a long and time-honoured tradition that aims at characterising
subvarieties of varieties of ordered algebras in terms of \textquotedblleft
forbidden configurations\textquotedblright, harking back to Dedekind's
celebrated result to the effect that the distributive subvariety of the
variety of lattices is the one whose members do not contain as subalgebras
$\mathbf{M}_{3}$ or $\mathbf{N}_{5}$, while the modular subvariety is the one
whose members do not contain $\mathbf{N}_{5}$. Other important results in the
same vein appear in the theory of ortholattices. For example, the benzene ring
$\mathbf{B}_{6}$:%

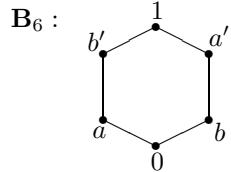
\begin{center}\begin{picture}(43,90)(0,0)
\put(20,30){\circle*{3}}
\put(0,40){\circle*{3}}
\put(40,40){\circle*{3}}
\put(0,65){\circle*{3}}
\put(40,65){\circle*{3}}
\put(20,75){\circle*{3}}
\put(0,40){\line(0,1){25}}
\put(40,40){\line(0,1){25}}
\put(18,21){$0$}
\put(18,78){$1$}
\put(20,30){\line(-2,1){20}}
\put(20,30){\line(2,1){20}}
\put(20,75){\line(-2,-1){20}}
\put(20,75){\line(2,-1){20}}
\put(-4,33){$a$}
\put(42,32){$b$}
\put(-6,66){$b^{\prime }$}
\put(40,68){$a^{\prime }$}
\put(-35,75){${\bf B}_6:$}\end{picture}\end{center}\vspace*{-15pt}

\noindent is a forbidden configuration for the orthomodular subvariety of the
variety of ortholattices; more precisely,%
\[
\mathbb{OML}=\left\{  \mathbf{L}\in\mathbb{OL}:\mathbf{B}_{6}\notin%
{\mathrm{S}}_{{\mathbb{I}} }(\mathbf{L})\right\}  \text{.}%
\]
Consequently:

\begin{lemma}
$(\mathbb{OML},V_{\mathbb{BI}}(\mathbf{B}_{6}))$ is a splitting pair in
$\mathrm{Subvar}\left(  \mathbb{OL}\right)  $.\label{splitol}
\end{lemma}

In this subsection, we intend to give a first, limited application of this
method, by means of a forbidden configuration consisting of a
\textquotedblleft paraorthomodular analogue\textquotedblright\ of
$\mathbf{B}_{6}$: the antiortholattice $\mathbf{D}_{2}\oplus\mathbf{B}%
_{6}\oplus\mathbf{D}_{2}$, hereafter denoted by $\mathbf{F}_{8}$, along with
any of its reducts, for the sake of brevity:%

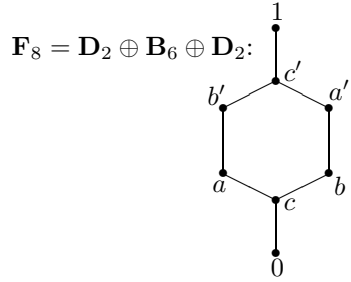
\begin{center}\begin{picture}(43,100)(0,0)
\put(20,10){\circle*{3}}
\put(20,30){\circle*{3}}
\put(0,40){\circle*{3}}
\put(40,40){\circle*{3}}
\put(0,65){\circle*{3}}
\put(40,65){\circle*{3}}
\put(20,75){\circle*{3}}
\put(20,95){\circle*{3}}
\put(20,10){\line(0,1){20}}
\put(20,95){\line(0,-1){20}}
\put(0,40){\line(0,1){25}}
\put(40,40){\line(0,1){25}}
\put(18,1){$0$}
\put(18,98){$1$}
\put(20,30){\line(-2,1){20}}
\put(20,30){\line(2,1){20}}
\put(20,75){\line(-2,-1){20}}
\put(20,75){\line(2,-1){20}}
\put(23,26){$c$}
\put(-4,33){$a$}
\put(42,32){$b$}
\put(-6,66){$b^{\prime }$}
\put(40,68){$a^{\prime }$}
\put(23,75){$c^{\prime }$}
\put(-80,85){${\bf F}_8={\bf D}_2\oplus {\bf B}_6\oplus {\bf D}_2$:}\end{picture}\end{center}

Since it has the $0$ meet--irreducible, the antiortholattice $\mathbf{F}_{8}$
satisfies SDM, thus $\mathbf{F}_{8}\in\mathbb{SAOL}$. The question arises
naturally as to which subvarieties ${\mathbb{V}}$ of $\mathbb{PBZL}^{\ast}$
are maximal with respect to the property that $\mathbf{F}_{8}\notin%
{\mathrm{S}}_{{\mathbb{I}} }({\mathbb{V}} )$, i.e. $\mathbf{F}_{8}%
\notin{\mathrm{S}}_{{\mathbb{I}} }(\mathbf{A})$ for any $\mathbf{A}%
\in{\mathbb{V}}$. This problem will be referred to as the \textquotedblleft%
$\mathbf{F}_{8}$ problem\textquotedblright. Although we will not give an
answer to this question, we provide a quasiequational characterisation of
paraorthomodular bounded involution lattices that do not contain
$\mathbf{F}_{8}$ as a bounded involution sublattice and we study the varieties
of PBZ$^{\ast}$ --lattices that contain the antiortholattice $\mathbf{F}_{8}$.

Clearly, for any $\mathbf{L},\mathbf{M}\in\mathbb{BI}$, we have:
$\mathbf{D}_{2}\oplus\mathbf{M}\oplus\mathbf{D}_{2}\in{\mathrm{S}%
}_{{\mathbb{I}}}(\mathbf{L})$ iff $\mathbf{D}_{2}\oplus\mathbf{M}%
\oplus\mathbf{D}_{2}\in{\mathrm{S}}_{\mathbb{BI}}(\mathbf{L})$. The
right-to-left direction is trivial, while, if $\mathbf{D}_{2}\oplus
\mathbf{M}\oplus\mathbf{D}_{2}\in{\mathrm{S}}_{{\mathbb{I}}}(\mathbf{L})$ and
$A=M\cup\{0,1\}$, then $\mathbf{D}_{2}\oplus\mathbf{M}\oplus\mathbf{D}%
_{2}\cong_{\mathbb{BI}}\mathbf{A}\in{\mathrm{S}}_{\mathbb{BI}}(\mathbf{L})$.
In particular, for any $\mathbf{A}\in\mathbb{BZL}$, we have that
$\mathbf{F}_{8}\in{\mathrm{S}}_{{\mathbb{I}}}(\mathbf{A})$ iff $\mathbf{F}%
_{8}\in{\mathrm{S}}_{\mathbb{BI}}(\mathbf{A})$; also, if $\mathbf{F}_{8}%
\in{\mathrm{S}}_{\mathbb{BZL}}(\mathbf{A})$, then $\mathbf{F}_{8}%
\in{\mathrm{S}}_{\mathbb{BI}}(\mathbf{A})$, while, if $\mathbf{A}$ is an
antiortholattice, then $\mathbf{F}_{8}\in{\mathrm{S}}_{\mathbb{BZL}%
}(\mathbf{A})$ iff $\mathbf{F}_{8}\in{\mathrm{S}}_{\mathbb{BI}}(\mathbf{A})$.

Observe what follows:

\begin{itemize}
\item no distributive PBZ$^{\ast}$ --lattice can contain $\mathbf{B}_{6}$ or
$\mathbf{F}_{8}$ as sublattices, in particular as sub-involution lattices;

\item since $\mathbf{B}_{6}$ is a sub-involution lattice of $\mathbf{F}_{8}$ and $\mathbf{B}_{6}$ is not a sub-involution lattice of any orthomodular lattice, no orthomodular lattice can contain $\mathbf{F}_{8}$ as a sub-involution lattice;

\item by the above, any subvariety ${\mathbb{V}}$ of $\mathbb{PBZL}^{\ast}$
such that ${\mathbb{V}}\subseteq\mathbb{DIST}\cup\mathbb{OML}$ satisfies
$\mathbf{F}_{8}\notin{\mathrm{S}}_{{\mathbb{I}}}({\mathbb{V}})$;

\item $\mathbf{F}_{8}\in\mathbb{SAOL}$, whence any subvariety ${\mathbb{V}}$
of $\mathbb{PBZL}^{\ast}$ such that $\mathbb{SAOL}\subseteq{\mathbb{V}}$
satisfies $\mathbf{F}_{8}\in{\mathrm{S}}_{{\mathbb{I}}}({\mathbb{V}})$.
\end{itemize}

Let us now consider the following quasiequations in the language of I--lattices:

\begin{flushleft}%
\begin{tabular}
[c]{ll}%
\mbox{\textcircled{\sc q}} & $x\leq y^{\prime}\ \&\ x^{\prime}\wedge
y^{\prime}\leq x\wedge y\Rightarrow x=y^{\prime}$\\
\mbox{\textcircled{\sc q}} $^{\prime}$ & $x^{\prime}\wedge(x^{\prime}\wedge
u)^{\prime}\leq x\wedge(x^{\prime}\wedge u)\Rightarrow u\leq x^{\prime}$%
\end{tabular}

\end{flushleft}

Note that \mbox{\textcircled{\sc q}}\ is equivalent to
\mbox{\textcircled{\sc q}} $^{\prime}$.

\begin{lemma}
If $\mathbf{A}\in{\mathbb{I}} $ and $a,b\in A$ are such that $a\leq b^{\prime
}$ and $a^{\prime}\wedge b^{\prime}\leq a\wedge b$, then $a\wedge a^{\prime
}=b\wedge b^{\prime}=a^{\prime}\wedge b^{\prime}=a\wedge b$.\label{posabc}
\end{lemma}

\begin{proof}
Let $c=a^{\prime}\wedge b^{\prime}$. Then $c\leq a\wedge b$ by the choice of
$a$ and $b$, therefore, since we also have $a\leq b^{\prime}$ and thus $b\leq
a^{\prime}$: $a\wedge a^{\prime}=a\wedge b^{\prime}\wedge a^{\prime}=a\wedge
c=c$; $b\wedge b^{\prime}=b\wedge a^{\prime}\wedge b^{\prime}=b\wedge c=c$;
$a\wedge b=a\wedge b^{\prime}\wedge b=a\wedge c=c$.
\end{proof}

\begin{lemma}
\label{b6f8}For any $\mathbf{A}\in\mathbb{PBI}$, we have:%
\[
\mathbf{B}_{6}\in{\mathrm{S}}_{{\mathbb{I}}}(\mathbf{A})\text{ iff }%
\mathbf{F}_{8}\in{\mathrm{S}}_{\mathbb{BI}}(\mathbf{A})\text{.}%
\]

\end{lemma}

\begin{proof}
The right-to-left direction is trivial. Now assume that $\mathbf{B}_{6}%
\in{\mathrm{S}}_{{\mathbb{I}}}(\mathbf{A})$, with $B_{6}=\{c,a,b,a^{\prime
},b^{\prime},c^{\prime}\}\subseteq A$, where $c=a\wedge b$ and $a<b^{\prime}$.
Assume ex absurdo that $c=0$, so that $a^{\prime}\wedge b^{\prime}=0$. Since
$\mathbf{A}$ is paraorthomodular, it follows that $a=b^{\prime}$, and we have
a contradiction. Therefore $c\neq0$, so, if we denote by
$L=\{0,c,a,b,a^{\prime},b^{\prime},c^{\prime},1\}$, then $\mathbf{F}_{8}%
\cong_{\mathbb{BI}}\mathbf{L}\in{\mathrm{S}}_{\mathbb{BI}}(\mathbf{A})$.
\end{proof}

\begin{proposition}
\label{pfrancesco}For any $\mathbf{A}\in{\mathbb{I}}$, we have:%
\[
\mathbf{A}\vDash\mbox{\textcircled{\sc q}}\text{ iff }\mathbf{B}_{6}%
\notin{\mathrm{S}}_{{\mathbb{I}}}(\mathbf{A})\text{.}%
\]

\end{proposition}

\begin{proof}
For the direct implication, assume that $\mathbf{B}_{6}\in{\mathrm{S}%
}(\mathbf{A})$, with $B_{6}=\{c,a,b,a^{\prime},$\linebreak$b^{\prime
},c^{\prime}\}\subseteq A$, where $c=a\wedge b$ and $a<b^{\prime}$. Then
$a\leq b^{\prime}$ and $a^{\prime}\wedge b^{\prime}=a\wedge b\leq a\wedge b$,
but $a\neq b^{\prime}$, hence $\mathbf{A}\nvDash\mbox{\textcircled{\sc q}}$.

For the converse, assume that $\mathbf{A}\nvDash\mbox{\textcircled{\sc q}}$,
so that there exist $a,b\in A$ with $a^{\prime}\wedge b^{\prime}\leq a\wedge
b$ and $a<b^{\prime}$, so $b<a^{\prime}$. Then, by Lemma \ref{posabc}, if we
denote by $c=a^{\prime}\wedge b^{\prime}$, then $c=a\wedge b=a\wedge
a^{\prime}=b\wedge b^{\prime}$. Since $a<b^{\prime}$, $a\wedge b\leq
a^{\prime}\vee b^{\prime}$; were it the case that $a\wedge b=a^{\prime}\vee
b^{\prime}$, we would have that $a^{\prime}\leq a^{\prime}\vee b^{\prime
}=a\wedge b\leq b$, a contradiction. Hence $c^{\prime}=(a\wedge b)^{\prime
}=a^{\prime}\vee b^{\prime}>a\wedge b=c$. Also, $a\vee b=(a^{\prime}\wedge
b^{\prime})^{\prime}=c^{\prime}$, $a\vee a^{\prime}=(a\wedge a^{\prime
})^{\prime}=c^{\prime}$ and $b\vee b^{\prime}=(b\wedge b^{\prime})^{\prime
}=c^{\prime}$. If we had $a\leq b$, then $a\leq b\wedge b^{\prime}=c=a\wedge
a^{\prime}\leq a$, hence $c=a\wedge a^{\prime}=a<b^{\prime}\leq a^{\prime
}\wedge b^{\prime}=c$, and we have a contradiction again. Similarly, $b\nleq
a$. Hence $a$ and $b$ are incomparable. Were it $a\leq a^{\prime}$, then
$c=a\wedge a^{\prime}=a$, which would lead to the same contradiction as above.
On the other hand, if $a^{\prime}\leq a$, then $c=a\wedge a^{\prime}%
=a^{\prime}>b\geq b\wedge b^{\prime}=c$, which gives us another contradiction.
Hence $a$ and $a^{\prime}$ are incomparable and so are, analogously, $b$ and
$b^{\prime}$. Therefore, if we denote by $L=\{c,a,b,a^{\prime},b^{\prime
},c^{\prime}\}$, then $\mathbf{B}_{6}\cong_{{\mathbb{I}}}\mathbf{L}%
\in{\mathrm{S}}_{{\mathbb{I}}}(\mathbf{A})$.
\end{proof}

\begin{theorem}
\label{francesco}For any $\mathbf{A}\in\mathbb{PBI}$, we have:%
\[
\mathbf{A}\vDash\mbox{\textcircled{\sc q}}\text{ iff }\mathbf{F}_{8}%
\notin{\mathrm{S}}_{\mathbb{BI}}(\mathbf{A})\text{.}%
\]

\end{theorem}

\begin{proof}
By Lemma \ref{b6f8} and Proposition \ref{pfrancesco}.
\end{proof}

\begin{example}
Here is an antiortholattice (in particular, a paraorthomodular BI--lattice)
$\mathbf{A}$ such that $\mathbf{F}_{8}\notin{\mathrm{S}}_{\mathbb{BI}%
}(\mathbf{A})$, but $\mathbf{F}_{8}\in{\mathrm{H}}_{\mathbb{BZL}}(\mathbf{A}%
)$, in particular $\mathbf{F}_{8}\in{\mathrm{S}}_{\mathbb{BI}}(({\mathrm{H}%
}_{\mathbb{BZL}}(\mathbf{A})))\subseteq{\mathrm{S}}_{\mathbb{BI}}({\mathrm{H}%
}_{\mathbb{BI}}(\mathbf{A}))$:%

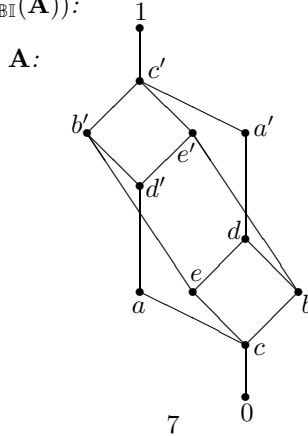
\begin{center}\hspace*{35pt}\begin{picture}(60,112)(0,0)
\put(40,-20){\circle*{3}}
\put(40,0){\circle*{3}}
\put(20,20){\circle*{3}}
\put(60,20){\circle*{3}}
\put(40,40){\circle*{3}}
\put(40,80){\circle*{3}}
\put(40,-20){\line(0,1){20}}
\put(40,40){\line(0,1){40}}
\put(40,0){\line(1,1){20}}
\put(40,0){\line(-1,1){20}}
\put(40,40){\line(1,-1){20}}
\put(40,40){\line(-1,-1){20}}
\put(40,0){\line(-2,1){40}}
\put(0,20){\circle*{3}}
\put(0,20){\line(0,1){40}}
\put(0,60){\line(1,1){20}}
\put(0,60){\line(-1,1){20}}
\put(0,100){\line(1,-1){20}}
\put(0,100){\line(-1,-1){20}}
\put(20,20){\line(-2,3){40}}
\put(60,20){\line(-2,3){40}}
\put(40,80){\line(-2,1){40}}
\put(0,60){\circle*{3}}
\put(-20,80){\circle*{3}}
\put(20,80){\circle*{3}}
\put(0,100){\circle*{3}}
\put(0,120){\circle*{3}}
\put(0,120){\line(0,-1){20}}
\put(-3,12){$a$}
\put(61,12){$b$}
\put(43,-4){$c$}
\put(3,101){$c^{\prime }$}
\put(-26,80){$b^{\prime }$}
\put(43,79){$a^{\prime }$}
\put(14,69){$e^{\prime }$}
\put(2,56){$d^{\prime }$}
\put(33,40){$d$}
\put(19,24){$e$}
\put(38,-29){$0$}
\put(-2,123){$1$}
\put(-50,105){${\bf A}$:}\end{picture}
\end{center}

The equivalence relation $\theta$ with cosets%
\[
\{0\},\{a\},\{c,e\},\{b,d\},\{b^{\prime},d^{\prime}\},\{c^{\prime},e^{\prime
}\},\{a^{\prime}\},\{1\}
\]
belongs to\ $\mathrm{Con}_{\mathbb{BI}01}(\mathbf{A})\subset\mathrm{Con}%
_{\mathbb{BZL}}(\mathbf{A})$ and $\mathbf{A}/\theta\cong\mathbf{F}_{8}$, but,
as announced above, $\mathbf{F}_{8}\notin{\mathrm{S}}_{\mathbb{BI}}%
(\mathbf{A})$.
\end{example}

\begin{corollary}
\mbox{\textcircled{\sc q}}\ is not an equational condition in $\mathbb{PBI}$
or $\mathbb{PBZL}^{\ast}$.
\end{corollary}

Now let us investigate the subvarieties of $\mathbb{PBZL}^{\ast}$ that contain
$\mathbf{F}_{8}$. We consider the following equation in the language of BZ--lattices:

\begin{flushleft}%
\begin{tabular}
[c]{rl}%
\mbox{D2OL$\vee$} & $(x\wedge x^{\prime})^{\sim}\vee(y\wedge y^{\prime}%
)^{\sim}\vee x\vee x^{\prime}\approx(x\wedge x^{\prime})^{\sim}\vee(y\wedge
y^{\prime})^{\sim}\vee y\vee y^{\prime}$%
\end{tabular}

\end{flushleft}

By \cite{rgcmfp}, $V_{\mathbb{BZL}}(\mathbb{AOL})$ is axiomatised by
\mbox{\textup{J0}}\ relative to $\mathbb{PBZL}^{\ast}$. By \cite{pbz5},
$V_{\mathbb{BZL}}(\mathbf{D}_{2}\oplus\mathbb{OL}\oplus\mathbf{D}_{2})$ is
axiomatised by \mbox{D2OL$\vee$}\ relative to $\mathbb{SAOL}$.

We use the following notation from \cite{pbz5}: for any $k,n,p\in{\mathbb{N}}$
and any equation $t\approx u$, where $t(x_{1},\ldots,x_{k},z_{1},\ldots
,z_{p})$ and $u(y_{1},\ldots,y_{n},z_{1},\ldots,z_{p})$ are terms in the
language of $\mathbb{BI}$ having the arities $k+p$, respectively $n+p$, and
$p$ common variables $z_{1},\ldots,z_{p}$, we denote by $m(t,u)$ the following
$(k+n)$--ary term in the language of $\mathbb{BZL}$:%
\[
m(t,u)(x_{1},\ldots,x_{k},y_{1},\ldots,y_{n},z_{1},\ldots,z_{p})=
\]
\vspace*{-7pt}%
\[
\bigvee_{i=1}^{k}(x_{i}\wedge x_{i}^{\prime})^{\sim}\vee\bigvee_{j=1}%
^{n}(y_{j}\wedge y_{j}^{\prime})^{\sim}\vee\bigvee_{h=1}^{p}(z_{h}\wedge
z_{h}^{\prime})^{\sim}\vee t(x_{1},\ldots,x_{k},z_{1},\ldots,z_{p}).
\]
Note that:%
\[
m(u,t)(x_{1},\ldots,x_{k},y_{1},\ldots,y_{n},z_{1},\ldots,z_{p})=
\]
\vspace*{-7pt}%
\[
\bigvee_{i=1}^{k}(x_{i}\wedge x_{i}^{\prime})^{\sim}\vee\bigvee_{j=1}%
^{n}(y_{j}\wedge y_{j}^{\prime})^{\sim}\vee\bigvee_{h=1}^{p}(z_{h}\wedge
z_{h}^{\prime})^{\sim}\vee u(y_{1},\ldots,y_{n},z_{1},\ldots,z_{p}).
\]

\begin{lemma}
\textrm{\cite[Corollary $6.14$]{pbz5}} For any ${\mathbb{C}} \subseteq
\mathbb{BI} $ and any ${\mathbb{D}} \subseteq\mathbb{PKA} $, $V_{\mathbb{BI}
}(\mathbf{D}_{2}\oplus{\mathbb{C}} \oplus\mathbf{D}_{2})=V_{\mathbb{BI}
}(\mathbf{D}_{2}\oplus V_{\mathbb{BI} }({\mathbb{C}} )\oplus\mathbf{D}_{2})$
and $V_{\mathbb{BZL} }(\mathbf{D}_{2}\oplus{\mathbb{D}} \oplus\mathbf{D}%
_{2})=V_{\mathbb{BZL} }(\mathbf{D}_{2}\oplus V_{\mathbb{BI} }({\mathbb{D}}
)\oplus\mathbf{D}_{2})$.\label{d2vard2}
\end{lemma}

\begin{proposition}
$V_{\mathbb{BI} }(\mathbf{F}_{8})=V_{\mathbb{BI} }(\mathbf{D}_{2}\oplus
V_{\mathbb{BI} }(\mathbf{B}_{6})\oplus\mathbf{D}_{2})$ and $V_{\mathbb{BZL}
}(\mathbf{F}_{8})=V_{\mathbb{BZL} }(\mathbf{D}_{2}\oplus V_{\mathbb{BI}
}(\mathbf{B}_{6})\oplus\mathbf{D}_{2})$.\label{f8varb6}
\end{proposition}

\begin{proof}
By Lemma \ref{d2vard2} and the fact that $\mathbf{F}_{8}=\mathbf{D}_{2}%
\oplus\mathbf{B}_{6}\oplus\mathbf{D}_{2}$.
\end{proof}

The following consequence of results from \cite{pbz5} shows that we can obtain
an axiomatisation for $V_{\mathbb{BZL}}(\mathbf{F}_{8})$ relative to
$\mathbb{PBZL}^{\ast}$ from an axiomatisation of $V_{\mathbb{BI}}%
(\mathbf{B}_{6})$ relative to $\mathbb{OL}$; note that any such axiomatisation
can be written with nonnullary terms over $\mathbb{BI}$, since $\mathbb{OL}$
satisfies the equations $x\vee x^{\prime}\approx1$ and $x\wedge x^{\prime
}\approx0$.

\begin{corollary}
$\{t_{i}\approx u_{i}:i\in I\}$ is an axiomatisation of $V_{\mathbb{BI}%
}(\mathbf{B}_{6})$ relative to $\mathbb{OL}$ such that, for each $i\in I$, the
terms $t_{i}$ and $u_{i}$ have nonzero arities iff $\{m(t_{i},u_{i})\approx
m(u_{i},t_{i}):i\in I\}\cup\{\mbox{\textup{J0}},\mbox{D2OL$\vee$}\}$ is an
axiomatisation of $V_{\mathbb{BZL}}(\mathbf{F}_{8})$ relative to
$\mathbb{PBZL}^{\ast}$.
\end{corollary}

\begin{proof}
By Proposition \ref{f8varb6}, the fact that $V_{\mathbb{BI} }(\mathbf{B}%
_{6})\subseteq\mathbb{OL} $ and \cite[Theorem $6.38.(ii)$]{pbz5}.
\end{proof}

\begin{theorem}
\textrm{\cite[Theorem $6.25$]{pbz5}} The operator ${\mathbb{V}} \mapsto
V_{\mathbb{BZL} }(\mathbf{D}_{2}\oplus{\mathbb{V}} \oplus\mathbf{D}_{2})$ is a
bounded lattice embedding from the lattice of subvarieties of $\mathbb{PKA} $
to the principal filter generated by $V_{\mathbb{BZL} }(\mathbf{D}_{3})$ in
the lattice of subvarieties of $\mathbb{SAOL} $.\label{thegenop}
\end{theorem}

\begin{corollary}
$(V_{\mathbb{BZL} }(\mathbf{D}_{2}\oplus\mathbb{OML} \oplus\mathbf{D}%
_{2}),V_{\mathbb{BZL} }(\mathbf{F}_{8}))$ is a splitting pair in the lattice
of subvarieties of $\mathbb{OL} $.
\end{corollary}

\begin{proof}
By Lemma \ref{splitol}, Proposition \ref{f8varb6} and Theorem \ref{thegenop}.
\end{proof}

\begin{proposition}
\begin{itemize}
\item $V_{\mathbb{BI} }(\mathbf{B}_{6})\subsetneq V_{\mathbb{BI} }%
(\mathbf{F}_{8})=V_{\mathbb{BI} }(\mathbf{D}_{n}\oplus\mathbf{F}_{8}%
\oplus\mathbf{D}_{n})$ for any $n\in{\mathbb{N}} ^{*}$;

\item $V_{\mathbb{BZL} }(\mathbf{F}_{8})\subsetneq V_{\mathbb{BZL}
}(\mathbf{D}_{2}\oplus\mathbf{F}_{8}\oplus\mathbf{D}_{2})=V_{\mathbb{BZL}
}(\mathbf{D}_{n}\oplus\mathbf{F}_{8}\oplus\mathbf{D}_{n})$ for any
$n\in{\mathbb{N}} \setminus\{0,1,2\}$.
\end{itemize}
\end{proposition}

\begin{proof}
By Proposition \ref{f8varb6}, the fact that $V_{\mathbb{BI} }(\mathbf{B}%
_{6})\subseteq\mathbb{OL} $, while $\mathbf{D}_{3}\in V_{\mathbb{BI}
}(\mathbf{F}_{8})$, and \cite[Corollary $6.23$]{pbz5}, we get that
$V_{\mathbb{BI} }(\mathbf{B}_{6})\subsetneq V_{\mathbb{BI} }(\mathbf{F}%
_{8})=V_{\mathbb{BI} }(\mathbf{D}_{2}\oplus\mathbf{F}_{8}\oplus\mathbf{D}%
_{2})$ and hence $V_{\mathbb{BI} }(\mathbf{F}_{8})=V_{\mathbb{BI} }%
(\mathbf{D}_{n}\oplus\mathbf{F}_{8}\oplus\mathbf{D}_{n})$ for any
$n\in{\mathbb{N}} ^{*}$. This, Theorem \ref{thegenop} and again Proposition
\ref{f8varb6} show that $V_{\mathbb{BZL} }(\mathbf{F}_{8})\subsetneq
V_{\mathbb{BZL} }(\mathbf{D}_{2}\oplus\mathbf{F}_{8}\oplus\mathbf{D}%
_{2})=V_{\mathbb{BZL} }(\mathbf{D}_{n}\oplus\mathbf{F}_{8}\oplus\mathbf{D}%
_{n})$ for any $n\in{\mathbb{N}} \setminus\{0,1,2\}$.
\end{proof}

\subsection{Covers in the Lattice of Subvarieties of $\mathbb{PBZL}^{\ast}$}

In this subsection, we continue the study of the lattice
$\mathrm{Subvar(\mathbb{PBZL}^{\ast})}$ of subvarieties of PBZ$^{\ast}$
--lattices, started in \cite{GLP1+, PBZ2, rgcmfp, pbzsums, pbz5}. We begin by
recapitulating a few known results.

\begin{lemma}
\begin{enumerate}
\item \label{subvarpbz0} \textrm{\cite[Subsection 5.3]{GLP1+}} $\mathbb{BA}$
is the unique atom of $\mathrm{Subvar(\mathbb{PBZL}^{\ast})}$.

\item \label{subvarpbz3} \textrm{\cite[Theorem 5.4.(2)]{GLP1+}} $\mathbb{BA}%
=\mathbb{OML}\cap V_{\mathbb{BZL} }(\mathbb{AOL})$.

\item \label{subvarpbz1} \textrm{\cite[Corollary 3.6]{BH}} The unique cover of
$\mathbb{BA} $ in the ideal $(\mathbb{OML} ]$ of\linebreak%
$\mathrm{Subvar(\mathbb{PBZL}^{\ast} )}$ is $V_{\mathbb{BZL} }(\mathbf{MO}%
_{2})$.

\item \label{subvarpbz2} \textrm{\cite[Theorem 5.5]{GLP1+}} For any
$\mathbf{L}\in\mathbb{PBZL}^{\ast}\setminus\mathbb{OML}$, we have
$\mathbf{D}_{3}\in{\mathrm{H}}{\mathrm{S}}(\mathbf{L})\subseteq
V_{\mathbb{BZL} }((\mathbf{L}))$, so the unique non--orthomodular cover of
$\mathbb{BA}$ in $\mathrm{Subvar(\mathbb{PBZL}^{\ast})}$ is $V_{\mathbb{BZL}
}(\mathbf{D}_{3})$.
\end{enumerate}

\label{subvarpbz}
\end{lemma}

By the above, in $\mathrm{Subvar(\mathbb{PBZL}^{\ast})}$ $V_{\mathbb{BZL}%
}(\mathbf{MO}_{2})$ and $V_{\mathbb{BZL}}(\mathbf{D}_{3})$ are the only covers
of $\mathbb{BA}$, and $\mathbb{OML}\vee V_{\mathbb{BZL}}(\mathbf{D}_{3})$ is
the unique cover of $\mathbb{OML}$.

\begin{lemma}
\textrm{\cite[Lemma 3.3.(1)]{PBZ2}} All subdirectly irreducible members
of\linebreak$V_{\mathbb{BZL} }(\mathbb{AOL} )$ belong to $\mathbb{AOL}
$.\label{sdirr}
\end{lemma}

\begin{lemma}
\textrm{\cite{pbz5}}\label{smallaolchains}

\begin{enumerate}
\item \label{smallaolchains1} $\mathbb{BA} =\mathbb{OML} \cap V_{\mathbb{BZL}
}(\mathbb{AOL} )=V_{\mathbb{BZL} }(\mathbf{D}_{2})\subsetneq V_{\mathbb{BZL}
}(\mathbf{D}_{3})\subsetneq V_{\mathbb{BZL} }(\mathbf{D}_{4})\subsetneq
V_{\mathbb{BZL} }(\mathbf{D}_{5})$.

\item \label{smallaolchains2} $Si(V_{\mathbb{BZL}}(\mathbf{D}_{3}%
))=V_{\mathbb{BZL}}(\mathbf{D}_{3})\cap\mathbb{AOL}={\mathrm{I}}%
_{\mathbb{BZL}}(\{\mathbf{D}_{1},\mathbf{D}_{2},\mathbf{D}_{3}\})$.
\end{enumerate}
\end{lemma}

We now prove the main result of this subsection.

\begin{theorem}
The only cover of $V_{\mathbb{BZL}}(\mathbf{D}_{3})$ in
$\mathrm{Subvar(\mathbb{PBZL}^{\ast})}$ included in\linebreak$V_{\mathbb{BZL}%
}(\mathbb{AOL})$ is $V_{\mathbb{BZL}}(\mathbf{D}_{4})$.
\end{theorem}

\begin{proof}
For any subvariety ${\mathbb{W}}$ of $V_{\mathbb{BZL}}(\mathbb{AOL})$ such
that $V_{\mathbb{BZL}}(\mathbf{D}_{3})\subsetneq{\mathbb{W}}$, there exists an
$\mathbf{A}\in Si({\mathbb{W}})\setminus Si(V_{\mathbb{BZL}}(\mathbf{D}%
_{3}))=({\mathbb{W}}\cap\mathbb{AOL})\setminus{\mathrm{I}}_{\mathbb{BZL}%
}(\{\mathbf{D}_{1},\mathbf{D}_{2},\mathbf{D}_{3}\})$ by Lemma \ref{sdirr} and
Lemma \ref{smallaolchains}.(\ref{smallaolchains2}), thus $\mathbf{A}$ is an
antiortholattice with $|A|>3$. Hence, there exists an $a\in A\setminus
\{0,1\}=A\setminus S_{\mathbb{BZL}}(\mathbf{A})$ with $a\neq a^{\prime}$, so
that $0<a\wedge a^{\prime}<a\vee a^{\prime}<1$. Therefore $\{0,a\wedge
a^{\prime},a\vee a^{\prime},1\}$ is the universe of a subalgebra of
$\mathbf{A}$ isomorphic to $\mathbf{D}_{4}$, i.e. $\mathbf{D}_{4}%
\in{\mathrm{S}}_{\mathbb{BZL}}(\mathbf{A})$, thus $V_{\mathbb{BZL}}%
(\mathbf{D}_{4})\subseteq V_{\mathbb{BZL}}(\mathbf{A})\subseteq{\mathbb{W}}$.
Since $V_{\mathbb{BZL}}(\mathbf{D}_{3})\subsetneq V_{\mathbb{BZL}}%
(\mathbf{D}_{4})$ by Lemma \ref{smallaolchains}.(\ref{smallaolchains1}), it
follows that $V_{\mathbb{BZL}}(\mathbf{D}_{4})$ is the only cover of
$V_{\mathbb{BZL}}(\mathbf{D}_{3})$ in $\mathrm{Subvar(V_{\mathbb{BZL}%
}(\mathbb{AOL}))}$, which is, of course, a convex sublattice of
$\mathrm{Subvar(\mathbb{PBZL}^{\ast})}$, thus $V_{\mathbb{BZL}}(\mathbf{D}%
_{4})$ is a cover of $V_{\mathbb{BZL}}(\mathbf{D}_{3})$ in
$\mathrm{Subvar(\mathbb{PBZL}^{\ast})}$.
\end{proof}

It remains open to determine whether $V_{\mathbb{BZL} }(\mathbf{D}_{4})$ is
the only cover of $V_{\mathbb{BZL} }(\mathbf{D}_{3})$ in
$\mathrm{Subvar(\mathbb{PBZL}^{\ast})}$. Recall, also, that $V_{\mathbb{BZL}
}(\mathbf{D}_{5})=\mathbb{SDM}\cap\mathbb{DIST}$ contains all antiortholattice
chains, i.e., all PBZ$^{\ast}$ --chains.

\begin{example}
Let us consider the following example of a PBZ$^{\ast}$ --lattice from
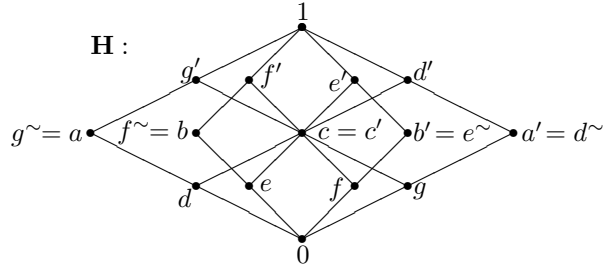
\cite{rgcmfp}:%

\begin{center}\begin{picture}(40,82)(0,0)
\put(-60,70){${\bf H}:$}
\put(20,0){\circle*{3}}
\put(20,40){\circle*{3}}
\put(20,80){\circle*{3}}
\put(20,80){\circle*{3}}
\put(-60,40){\circle*{3}}
\put(100,40){\circle*{3}}
\put(-20,40){\circle*{3}}
\put(60,40){\circle*{3}}
\put(0,60){\circle*{3}}
\put(-20,60){\circle*{3}}
\put(40,60){\circle*{3}}
\put(60,60){\circle*{3}}
\put(0,20){\circle*{3}}
\put(-20,20){\circle*{3}}
\put(40,20){\circle*{3}}
\put(60,20){\circle*{3}}
\put(-27,12){$d$}
\put(4,18){$e$}
\put(4,59){$f^{\prime }$}
\put(30,56){$e^{\prime }$}
\put(-26,61){$g^{\prime }$}
\put(62,60){$d^{\prime }$}
\put(30,16){$f$}
\put(62,17){$g$}
\put(18,-9){$0$}
\put(18,83){$1$}
\put(-50,38){$f^{\sim }\!\!=b$}
\put(26,38){$c=c^{\prime }$}
\put(62,37){$b^{\prime }\!=e^{\sim }$}
\put(103,37){$a^{\prime }\!=d^{\sim }$}
\put(-90,37){$g^{\sim }\!\!=a$}
\put(20,0){\line(-2,1){80}}
\put(20,0){\line(2,1){80}}
\put(20,0){\line(-1,1){40}}
\put(20,0){\line(1,1){40}}
\put(20,80){\line(-2,-1){80}}
\put(20,80){\line(2,-1){80}}
\put(20,80){\line(-1,-1){40}}
\put(20,80){\line(1,-1){40}}
\put(0,60){\line(1,-1){40}}
\put(-20,60){\line(2,-1){80}}
\put(40,60){\line(-1,-1){40}}
\put(60,60){\line(-2,-1){80}}
\end{picture}\end{center}

Note that:

\begin{itemize}
\item $\mathbf{H}\vDash\{SDM,SK\}$, thus $\mathbb{OML} \vee V_{\mathbb{BZL}
}(\mathbf{H})\vDash\{SDM,SK\}$ since $\mathbb{OML} \vDash\{SDM,SK\}$;

\item $\mathbf{H}\nvDash J2$, thus $\mathbf{H}\notin\mathbb{OML} \vee
V_{\mathbb{BZL} }(\mathbb{AOL} )\vDash J2$, in particular $\mathbf{H}%
\notin\mathbb{OML} \vee V_{\mathbb{BZL} }(\mathbf{D}_{3})$;

\item $\mathbf{D}_{3}\in{\mathrm{S}} (\mathbf{H})$, hence $\mathbb{OML} \vee
V_{\mathbb{BZL} }(\mathbf{D}_{3})\subseteq\mathbb{OML} \vee V_{\mathbb{BZL}
}(\mathbf{H})$, therefore $\mathbb{OML} \vee V_{\mathbb{BZL} }(\mathbf{D}%
_{3})\subsetneq\mathbb{OML} \vee V_{\mathbb{BZL} }(\mathbf{H})$ by the above;

\item since $\mathbb{OML} \vee V_{\mathbb{BZL} }(\mathbf{H})\vDash SK$ and
$\mathbf{D}_{4}\nvDash SK$, we have $\mathbf{D}_{4}\notin\mathbb{OML} \vee
V_{\mathbb{BZL} }(\mathbf{H})$, hence $\mathbf{D}_{4}$ does not belong to
every proper supervariety of $\mathbb{OML} \vee V_{\mathbb{BZL} }%
(\mathbf{D}_{3})$.
\end{itemize}

$\mathbf{H}\vDash\{SDM,SK\}$, $\mathbf{H}\nvDash J2$ and $\mathbb{OML}\vee
V_{\mathbb{BZL} }(\mathbb{AOL})\vDash J2$, hence $\mathbf{H}\in(\mathbb{SDM}%
\cap\mathbb{SK})\setminus(\mathbb{OML}\vee V_{\mathbb{BZL} }(\mathbb{AOL}))$,
thus $\mathbb{SDM}\cap\mathbb{SK}\nsubseteq\mathbb{OML}\vee V_{\mathbb{BZL}
}(\mathbb{AOL})$. $\mathbb{AOL}\nvDash SDM$ and $\mathbb{AOL}\nvDash SK$, thus
$\mathbb{AOL}\nsubseteq\mathbb{SDM}$ and $\mathbb{AOL}\nsubseteq\mathbb{SK}$,
in particular $\mathbb{OML}\vee V_{\mathbb{BZL} }(\mathbb{AOL})\nsubseteq
\mathbb{SDM}\cap\mathbb{SK}$. Therefore $\mathbb{SDM}\cap\mathbb{SK}%
||\mathbb{OML}\vee V_{\mathbb{BZL} }(\mathbb{AOL})$. Now let ${\mathbb{V}%
}=V_{\mathbb{BZL} }({MO}_{2})\vee V_{\mathbb{BZL} }(\mathbf{D}_{3}%
)\subseteq\mathbb{SDM}\cap\mathbb{SK}$. $\mathbf{D}_{3}\notin\mathbb{OML}$,
thus ${\mathbb{V}}\nsubseteq\mathbb{OML}$. ${MO}_{2}\notin V_{\mathbb{BZL}
}(\mathbb{AOL})$, thus ${\mathbb{V}}\nsubseteq V_{\mathbb{BZL} }%
(\mathbb{AOL})$. Finally, ${\mathbb{V}}$ satisfies the modular law, while both
$\mathbb{OML}$ and $V_{\mathbb{BZL} }(\mathbb{AOL})$ fail it, hence
$\mathbb{OML}\nsubseteq{\mathbb{V}}$ and $V_{\mathbb{BZL} }(\mathbb{AOL}%
)\nsubseteq{\mathbb{V}}$. Therefore $\mathbb{OML}||{\mathbb{V}}%
||V_{\mathbb{BZL} }(\mathbb{AOL})$.
\end{example}

We list hereafter a few problems that remain open at the time of writing:

\begin{itemize}
\item Is $\mathbb{OML}\vee V_{\mathbb{BZL} }(\mathbf{D}_{4})$ a successor of
$\mathbb{OML}\vee V_{\mathbb{BZL} }(\mathbf{D}_{3})$ in
$\mathrm{Subvar(\mathbb{PBZL}^{\ast})}$? Is it its only successor?

\item Is $\mathrm{Subvar(\mathbb{PBZL}^{\ast})}$ strongly atomic? If so, then
$\mathbb{OML}\vee V_{\mathbb{BZL} }(\mathbf{H})$ includes a successor of
$\mathbb{OML}\vee V_{\mathbb{BZL} }(\mathbf{D}_{3})$ which differs from
$\mathbb{OML}\vee V_{\mathbb{BZL} }(\mathbf{D}_{4})$.
\end{itemize}

\subsection{Subdirect Products and Varieties of PBZ$^{\ast}$ --lattices}

Let ${\mathbb{V}}$ and ${\mathbb{W}}$ be varieties of the same type.
Obviously, if ${\mathbb{V}}$ and ${\mathbb{W}}$ are incomparable, then there
exist $\mathbf{A}\in({\mathbb{V}}\vee{\mathbb{W}})\setminus{\mathbb{V}}$ and
$\mathbf{B}\in({\mathbb{V}}\vee{\mathbb{W}})\setminus{\mathbb{W}}$, so that
$\mathbf{A}\times\mathbf{B}\in({\mathbb{V}}\vee{\mathbb{W}})\setminus
({\mathbb{V}}\cup{\mathbb{W}})$ and thus ${\mathbb{V}}\cup{\mathbb{W}%
}\subsetneq{\mathbb{V}}\vee{\mathbb{W}}$. Recall that the \emph{subdirect
product} of ${\mathbb{V}}$ and ${\mathbb{W}}$ is the class, denoted by
${\mathbb{V}}\times_{s}{\mathbb{W}}$, whose members are isomorphic images of
subdirect products of a member of ${\mathbb{V}}$ and a member of ${\mathbb{W}%
}$. Clearly, ${\mathbb{V}}\cup{\mathbb{W}}\subseteq{\mathbb{V}}\times
_{s}{\mathbb{W}}\subseteq{\mathbb{V}}\vee{\mathbb{W}}$, so that%
\[
Si({\mathbb{V}})\cup Si({\mathbb{W}})=Si({\mathbb{V}}\cup{\mathbb{W}%
})\subseteq Si({\mathbb{V}}\times_{s}{\mathbb{W}})\subseteq Si({\mathbb{V}%
}\vee{\mathbb{W}}).
\]
For any $\mathbf{M}\in Si({\mathbb{V}}\times_{s}{\mathbb{W}})$, $\mathbf{M}$
is a subdirect product of an $\mathbf{A}\in{\mathbb{V}}$ and a $\mathbf{B}%
\in{\mathbb{W}}$, so that $\mathbf{A}$ is trivial, case in which
$\mathbf{M}\in Si({\mathbb{W}})$, or $\mathbf{B}$ is trivial, case in which
$\mathbf{M}\in Si({\mathbb{V}})$. Thus $Si({\mathbb{V}}\times_{s}{\mathbb{W}%
})\subseteq Si({\mathbb{V}})\cup Si({\mathbb{W}})$, hence $Si({\mathbb{V}%
}\times_{s}{\mathbb{W}})=Si({\mathbb{V}})\cup Si({\mathbb{W}})$. Since
${\mathbb{V}}\times_{s}\mathbb{W}\subseteq{\mathbb{V}}\vee{\mathbb{W}}$, we
get that the following equivalence holds: ${\mathbb{V}}\vee{\mathbb{W}%
}={\mathbb{V}}\times_{s}{\mathbb{W}}$ iff $Si({\mathbb{V}}\vee{\mathbb{W}%
})=Si({\mathbb{V}})\cup Si({\mathbb{W}})$.

Sufficient Maltsev-type conditions for the equivalence ${\mathbb{V}}%
\vee{\mathbb{W}}={\mathbb{V}}\times_{s}{\mathbb{W}}$ to hold are available in
the literature: see \cite{pl, tkfp, klp}. These contributions are all inspired
by the celebrated result by Gr\"{a}tzer, Lakser and P\l onka according to
which two independent similar varieties ${\mathbb{V}}$ and ${\mathbb{W}}$ are
such that every member of ${\mathbb{V}}\vee{\mathbb{W}}$ is isomorphic to
the\emph{ direct }product of a member of ${\mathbb{V}}$ and a member of
${\mathbb{W}}$ \cite{glp}. Of course, the notion of independence is of limited
use in the context of PBZ$^{\ast}$ --lattices, since ${\mathbb{BA}}$ is the
unique atom in $\mathrm{Subvar(\mathbb{PBZL}^{\ast})}$ and thus there are no
two nontrivial disjoint (hence, no two independent) varieties of PBZ$^{\ast}$
--lattices. The investigation of subdirect products of varieties of
PBZ$^{\ast}$ --lattices, however, can be carried out with more ad hoc methods,
yielding useful information on joins of specific subvarieties.

If ${\mathbb{V}}\vee{\mathbb{W}}={\mathbb{V}}\times_{s}{\mathbb{W}}$ and
${\mathbb{U}}$ is a variety of the same type as ${\mathbb{V}}$ and
${\mathbb{W}}$, then $({\mathbb{U}}\cap{\mathbb{V}})\times_{s}({\mathbb{U}%
}\cap{\mathbb{W}})\subseteq({\mathbb{U}}\cap{\mathbb{V}})\vee({\mathbb{U}}%
\cap{\mathbb{W}})\subseteq{\mathbb{U}}\cap({\mathbb{V}}\vee{\mathbb{W}})$ and%
\[%
\begin{array}
[c]{ll}%
Si({\mathbb{U}}\cap({\mathbb{V}}\vee{\mathbb{W}})) & =Si({\mathbb{U}})\cap
Si({\mathbb{V}}\vee{\mathbb{W}})\\
& =Si({\mathbb{U}})\cap(Si({\mathbb{V}})\cup Si({\mathbb{W}}))\\
& =(Si({\mathbb{U}})\cap Si({\mathbb{V}}))\cup(Si({\mathbb{U}})\cap
Si({\mathbb{W}}))\\
& =Si({\mathbb{U}}\cap{\mathbb{V}})\cup Si({\mathbb{U}}\cap{\mathbb{W}})\\
& =Si(({\mathbb{U}}\cap{\mathbb{V}})\times_{s}({\mathbb{U}}\cap{\mathbb{W}})),
\end{array}
\]
hence ${\mathbb{U}}\cap({\mathbb{V}}\vee{\mathbb{W}})=({\mathbb{U}}%
\cap{\mathbb{V}})\vee({\mathbb{U}}\cap{\mathbb{W}})=({\mathbb{U}}%
\cap{\mathbb{V}})\times_{s}({\mathbb{U}}\cap{\mathbb{W}})$. For instance,
since $\mathbb{OML}\vee V_{\mathbb{BZL}}(\mathbb{AOL})=\mathbb{OML}\times
_{s}V_{\mathbb{BZL}}(\mathbb{AOL})$ (see Lemma \ref{subdirirred} below), it
follows that%
\begin{align*}
\mathbb{SDM}\cap(\mathbb{OML}\vee V_{\mathbb{BZL}}(\mathbb{AOL}))  &
=(\mathbb{SDM}\cap\mathbb{OML})\vee(\mathbb{SDM}\cap V_{\mathbb{BZL}%
}(\mathbb{AOL}))\\
&  =\mathbb{OML}\vee\mathbb{SAOL}=\mathbb{OML}\times_{s}\mathbb{SAOL}.
\end{align*}
As a consequence of the above, if ${\mathbb{V}}\vee{\mathbb{W}}={\mathbb{V}%
}\times_{s}{\mathbb{W}}$ and $\mathrm{Subvar}({\mathbb{V}})$ and
$\mathrm{Subvar}({\mathbb{W}})$ are distributive, then $\mathrm{Subvar}%
({\mathbb{V}}\vee{\mathbb{W}})$ is distributive.

\begin{problem}
If ${\mathbb{V}}\vee{\mathbb{W}}={\mathbb{V}}\times_{s}{\mathbb{W}}$,
${\mathbb{C}}$ is a subvariety of ${\mathbb{V}}$ and ${\mathbb{D}}$ is a
subvariety of ${\mathbb{W}}$, under what conditions does it follow that
${\mathbb{C}}\vee{\mathbb{D}}={\mathbb{C}}\times_{s}{\mathbb{D}}$? Does the
condition that ${\mathbb{C}}\cap{\mathbb{D}}={\mathbb{V}}\cap{\mathbb{W}}$
suffice? A partial answer to this question is given by Lemma \ref{sprodvar} below.
\end{problem}

If ${\mathbb{V}}\vee{\mathbb{W}}={\mathbb{V}}\times_{s}{\mathbb{W}}$ and
${\mathbb{U}}$ is a subvariety of ${\mathbb{V}}$, then ${\mathbb{U}}%
\vee{\mathbb{W}}$ is a subvariety of ${\mathbb{V}}\vee{\mathbb{W}}$, so that%
\[%
\begin{array}
[c]{ll}%
Si({\mathbb{U}}\vee{\mathbb{W}}) & =({\mathbb{U}}\vee{\mathbb{W}})\cap
Si({\mathbb{V}}\vee{\mathbb{W}})\\
& =({\mathbb{U}}\vee{\mathbb{W}})\cap(Si({\mathbb{V}})\cup Si({\mathbb{W}}))\\
& =(({\mathbb{U}}\vee{\mathbb{W}})\cap Si({\mathbb{V}}))\cup(({\mathbb{U}}%
\vee{\mathbb{W}})\cap Si({\mathbb{W}}))\\
& =(({\mathbb{U}}\vee{\mathbb{W}})\cap Si({\mathbb{V}}))\cup Si({\mathbb{W}}).
\end{array}
\]

\begin{lemma}
\label{sprodvar}Let ${\mathbb{V}}$ and ${\mathbb{W}}$ be varieties of a
similarity type $\tau$, ${\mathbb{U}}$ a subvariety of ${\mathbb{V}}$ and
$\Gamma$ a set of equations over $\tau$ such that ${\mathbb{V}}\vee
{\mathbb{W}}={\mathbb{V}}\times_{s}{\mathbb{W}}$, ${\mathbb{W}}\vDash\Gamma$
and ${\mathbb{U}}=\{\mathbf{A}\in{\mathbb{V}}:\mathbf{A}\vDash\Gamma\}$. Then:

\begin{itemize}
\item ${\mathbb{U}} \vee{\mathbb{W}} ={\mathbb{U}} \times_{s}{\mathbb{W}}
=\{\mathbf{A}\in{\mathbb{V}} \vee{\mathbb{W}} :\mathbf{A}\vDash\Gamma\}$;

\item ${\mathbb{U}}={\mathbb{V}}$ iff ${\mathbb{U}}\vee{\mathbb{W}%
}={\mathbb{V}}\vee{\mathbb{W}}$.
\end{itemize}
\end{lemma}

\begin{proof}
Of course, $Si({\mathbb{U}})\cup Si({\mathbb{W}})\subseteq Si({\mathbb{U}}%
\vee{\mathbb{W}})$. For all $\mathbf{A}\in Si({\mathbb{U}}\vee{\mathbb{W}})$,
we have: $\mathbf{A}\in Si({\mathbb{V}}\vee{\mathbb{W}})=Si({\mathbb{V}})\cup
Si({\mathbb{W}})$ and $\mathbf{A}\vDash\Gamma$, so that either $\mathbf{A}\in
Si({\mathbb{W}})$ or $\mathbf{A}\in Si({\mathbb{V}})\subset{\mathbb{V}}$ and
$\mathbf{A}\vDash\Gamma$, the latter of which implies that $\mathbf{A}\in
Si({\mathbb{V}})\cap{\mathbb{U}}=Si({\mathbb{U}})$. Therefore $Si({\mathbb{U}%
}\vee{\mathbb{W}})=Si({\mathbb{U}})\cup Si({\mathbb{W}})$, thus ${\mathbb{U}%
}\vee{\mathbb{W}}={\mathbb{U}}\times_{s}{\mathbb{W}}$. We have that:%
\[%
\begin{array}
[c]{ll}%
Si(\{\mathbf{A}\in{\mathbb{V}}\vee{\mathbb{W}}:\mathbf{A}\vDash\Gamma\}) &
=\{\mathbf{A}\in Si({\mathbb{V}}\vee{\mathbb{W}}):\mathbf{A}\vDash\Gamma\}\\
& =\{\mathbf{A}\in Si({\mathbb{V}})\cup Si({\mathbb{W}}):\mathbf{A}%
\vDash\Gamma\}\\
& =\{\mathbf{A}\in Si({\mathbb{V}}):\mathbf{A}\vDash\Gamma\}\cup
Si({\mathbb{W}})\\
& =Si(\{\mathbf{A}\in{\mathbb{V}}:\mathbf{A}\vDash\Gamma\})\cup Si({\mathbb{W}%
})\\
& =Si({\mathbb{U}})\cup Si({\mathbb{W}})=Si({\mathbb{U}}\vee{\mathbb{W}}),
\end{array}
\]
hence ${\mathbb{U}}\vee{\mathbb{W}}=\{\mathbf{A}\in{\mathbb{V}}\vee
{\mathbb{W}}:\mathbf{A}\vDash\Gamma\}$.

Trivially, ${\mathbb{U}}={\mathbb{V}}$ implies ${\mathbb{U}}\vee{\mathbb{W}%
}={\mathbb{V}}\vee{\mathbb{W}}$. Conversely, if ${\mathbb{V}}\vee{\mathbb{W}%
}={\mathbb{U}}\vee{\mathbb{W}}=\{\mathbf{A}\in{\mathbb{V}}\vee{\mathbb{W}%
}:\mathbf{A}\vDash\Gamma\}$, then ${\mathbb{V}}\vee{\mathbb{W}}\vDash\Gamma$,
thus ${\mathbb{V}}\vDash\Gamma$, hence ${\mathbb{U}}=\{\mathbf{A}%
\in{\mathbb{V}}:\mathbf{A}\vDash\Gamma\}={\mathbb{V}}$.
\end{proof}

\begin{lemma}
\textrm{\cite{rgcmfp}}\label{subdirirred} All subdirectly irreducible members
of $\mathbb{OML}\vee V_{\mathbb{BZL}}(\mathbb{AOL})$ belong to $\mathbb{OML}%
\cup\mathbb{AOL}$, in particular $\mathbb{OML}\vee V_{\mathbb{BZL}%
}(\mathbb{AOL})=\mathbb{OML}\times_{s}V_{\mathbb{BZL}}(\mathbb{AOL})$.
\end{lemma}

We can derive from the above the following result from \cite{rgcmfp}:

\begin{proposition}
\begin{itemize}
\item $\mathbb{OML} \vee\mathbb{SAOL} =\mathbb{OML} \times_{s}\mathbb{SAOL} $
and $\mathbb{OML} \vee V_{\mathbb{BZL} }(\mathbf{D}_{3})=\mathbb{OML}
\times_{s}V_{\mathbb{BZL} }(\mathbf{D}_{3})$;

\item $\mathbb{OML} \vee V_{\mathbb{BZL} }(\mathbf{D}_{3})\subsetneq
\mathbb{OML} \vee\mathbb{SAOL} \subsetneq\mathbb{OML} \vee V_{\mathbb{BZL}
}(\mathbb{AOL} )$.
\end{itemize}

\label{d3saol}
\end{proposition}

\begin{proof}
Recall from \cite[Corollary 3.3]{PBZ2} that $V_{\mathbb{BZL}}(\mathbf{D}%
_{3})=\{\mathbf{A}\in V_{\mathbb{BZL}}(\mathbb{AOL}):\mathbf{A}\vDash
\{SDM,SK\}\}$. Now apply the fact that $\mathbb{OML}\vDash\{SDM,SK\}$ and
Lemmas \ref{subdirirred} and \ref{sprodvar} to obtain first that
$\mathbb{OML}\vee\mathbb{SAOL}=\mathbb{OML}\times_{s}\mathbb{SAOL}$, then that
$\mathbb{OML}\vee V_{\mathbb{BZL}}(\mathbf{D}_{3})=\mathbb{OML}\times
_{s}V_{\mathbb{BZL}}(\mathbf{D}_{3})$. Recall that $\mathbf{D}_{5}%
\in\mathbb{SAOL}\setminus V_{\mathbb{BZL}}(\mathbf{D}_{3})$, which is easily
noticed from the fact that $\mathbf{D}_{5}\nvDash SK$. The antiortholattice
$\mathbf{D}_{2}^{2}\oplus\mathbf{D}_{2}^{2}\in V_{\mathbb{BZL}}(\mathbb{AOL}%
)\setminus\mathbb{SAOL}$. Hence $V_{\mathbb{BZL}}(\mathbf{D}_{3}%
)\subsetneq\mathbb{SAOL}\subsetneq V_{\mathbb{BZL}}(\mathbb{AOL})$, thus
$\mathbb{OML}\vee V_{\mathbb{BZL}}(\mathbf{D}_{3})\subsetneq\mathbb{OML}%
\vee\mathbb{SAOL}\subsetneq\mathbb{OML}\vee V_{\mathbb{BZL}}(\mathbb{AOL})$ by
Lemma \ref{sprodvar} and the above.
\end{proof}

Let us consider the identities:

\begin{flushleft}%
\begin{tabular}
[c]{ll}%
$WDSDM$ & $(x\wedge(y\vee z))^{\sim}\approx(x\wedge y)^{\sim}\wedge(x\wedge
z)^{\sim}$\\
$DIST\vee^{\sim}$ & $(x\vee x^{\sim})\wedge(y\vee y^{\sim}\vee z\vee z^{\sim
})\approx$\\
& $((x\vee x^{\sim})\wedge(y\vee y^{\sim}))\vee((x\vee x^{\sim})\wedge(z\vee
z^{\sim}))$\\
$WDIST\vee^{\sim}$ & $((x\vee x^{\sim})\wedge(y\vee y^{\sim}\vee z\vee
z^{\sim}))^{\sim}\approx$\\
& $(((x\vee x^{\sim})\wedge(y\vee y^{\sim}))\vee((x\vee x^{\sim})\wedge(z\vee
z^{\sim})))^{\sim}$%
\end{tabular}

\end{flushleft}

Note that $WDSDM$ implies $WDIST\vee^{\sim}$ and $DIST\vee^{\sim}$ implies
$WDIST\vee^{\sim}$. Also, recall from \cite{PBZ2, pbz5} that $V_{\mathbb{BZL}%
}(\mathbf{D}_{5})=\mathbb{SAOL}\cap\mathbb{DIST}$.

\begin{proposition}
\label{distnsaol}$V_{\mathbb{BZL}}(\mathbf{D}_{5})=\mathbb{SAOL}%
\cap\mathbb{DIST}\subsetneq\mathbb{SAOL},\mathbb{DIST}\subsetneq
\mathbb{SAOL}\vee\mathbb{DIST}\subsetneq V_{\mathbb{BZL}}(\mathbb{AOL})$.
\end{proposition}

\begin{proof}
Observe that the identity $WDSDM$ is satisfied both in $\mathbb{SAOL}$ and in
$\mathbb{DIST}$. The antiortholattice on $\mathbf{M}_{3}\oplus\mathbf{M}_{3}$
fails WDSDM, because, if $a,b,c$ are its three atoms, then $(a\wedge(b\vee
c))^{\sim}=a^{\sim}=0$, yet $(a\wedge b)^{\sim}\wedge(a\wedge c)^{\sim
}=0^{\sim}\wedge0^{\sim}=1$. Hence $\mathbf{M}_{3}\oplus\mathbf{M}_{3}%
\in\mathbb{AOL}\setminus(\mathbb{SAOL}\vee\mathbb{DIST})\subseteq
V_{\mathbb{BZL}}(\mathbb{AOL})\setminus(\mathbb{SAOL}\vee\mathbb{DIST})$. The
antiortholattice $\mathbf{D}_{2}\oplus\mathbf{M}_{3}\oplus\mathbf{D}_{2}%
\in\mathbb{SAOL}\setminus\mathbb{DIST}$, while the antiortholattice
$\mathbf{D}_{2}^{2}\oplus\mathbf{D}_{2}^{2}\in\mathbb{DIST}\setminus
\mathbb{SAOL}$, hence $\mathbb{SAOL}$ and $\mathbb{DIST}$ are incomparable,
thus $\mathbb{SAOL}\cap\mathbb{DIST}\subsetneq\mathbb{SAOL},\mathbb{DIST}%
\subsetneq\mathbb{SAOL}\vee\mathbb{DIST}$.
\end{proof}

\begin{proposition}
\begin{itemize}
\item $\mathbb{OML} \vee\mathbb{DIST} =\mathbb{OML} \times_{s}\mathbb{DIST} $
and $\mathbb{OML} \vee V_{\mathbb{BZL} }(\mathbf{D}_{5})=\mathbb{OML}
\times_{s}V_{\mathbb{BZL} }(\mathbf{D}_{5})$;

\item $\mathbb{OML} \vee V_{\mathbb{BZL} }(\mathbf{D}_{3})\subsetneq
\mathbb{OML} \vee V_{\mathbb{BZL} }(\mathbf{D}_{5})=\mathbb{OML}
\vee(\mathbb{SAOL} \cap\mathbb{DIST} )=(\mathbb{OML} \vee\mathbb{SAOL}
)\cap(\mathbb{OML} \vee\mathbb{DIST} )\subsetneq\mathbb{OML} \vee\mathbb{SAOL}
,\mathbb{OML} \vee\mathbb{DIST} \subsetneq\mathbb{OML} \vee\mathbb{SAOL}
\vee\mathbb{DIST} \subsetneq\mathbb{OML} \vee V_{\mathbb{BZL} }(\mathbb{AOL}
)$, in particular the varieties $\mathbb{OML} \vee\mathbb{SAOL} $ and
$\mathbb{OML} \vee\mathbb{DIST} $ are incomparable.
\end{itemize}
\end{proposition}

\begin{proof}
Note that $\mathbb{OML}\vDash DIST\vee^{\sim}$ and that, in $\mathbb{AOL}$,
$DIST\vee^{\sim}$ is equivalent to $DIST$, that is $\mathbb{DIST}%
\cap\mathbb{AOL}=\{\mathbf{A}\in\mathbb{AOL}:\mathbf{A}\vDash DIST\vee^{\sim
}\}$. The latter, along with the fact that $\mathbb{DIST}$ is a subvariety of
$V_{\mathbb{BZL}}(\mathbb{AOL})$ and Lemma \ref{sdirr}, give us:%
\[%
\begin{array}
[c]{ll}%
Si(\mathbb{DIST}) & =\mathbb{DIST}\cap Si(V_{\mathbb{BZL}}(\mathbb{AOL}))\\
& =\mathbb{DIST}\cap Si(\mathbb{AOL})\\
& =Si(\mathbb{DIST}\cap\mathbb{AOL})\\
& =Si(\{\mathbf{A}\in\mathbb{AOL}:\mathbf{A}\vDash DIST\vee^{\sim}\})\\
& =Si(\{\mathbf{A}\in V_{\mathbb{BZL}}(\mathbb{AOL}):\mathbf{A}\vDash
DIST\vee^{\sim}\}),
\end{array}
\]
therefore $\mathbb{DIST}=\{\mathbf{A}\in V_{\mathbb{BZL}}(\mathbb{AOL}%
):\mathbf{A}\vDash DIST\vee^{\sim}\}$. By Lemmas \ref{subdirirred} and
\ref{sprodvar}, it follows that $\mathbb{OML}\vee\mathbb{DIST}=\mathbb{OML}%
\times_{s}\mathbb{DIST}$. By the above, $\mathbb{OML}\vDash\{SDM,DIST\vee
^{\sim}\}$ and $V_{\mathbb{BZL}}(\mathbf{D}_{5})=\mathbb{SAOL}\cap
\mathbb{DIST}=\{\mathbf{A}\in V_{\mathbb{BZL}}(\mathbb{AOL}):\mathbf{A}%
\vDash\{SDM,DIST\vee^{\sim}\}\}$, hence $\mathbb{OML}\vee V_{\mathbb{BZL}%
}(\mathbf{D}_{5})=\mathbb{OML}\times_{s}V_{\mathbb{BZL}}(\mathbf{D}_{5})$ by
Lemmas \ref{subdirirred} and \ref{sprodvar}. By the above, Propositions
\ref{d3saol} and \ref{distnsaol} and again Lemma \ref{sprodvar}, it follows
that:%
\begin{align*}
\mathbb{OML}\vee V_{\mathbb{BZL}}(\mathbf{D}_{3})  &  \subsetneq
\mathbb{OML}\vee V_{\mathbb{BZL}}(\mathbf{D}_{5})\\
&  \subsetneq\mathbb{OML}\vee\mathbb{SAOL},\mathbb{OML}\vee\mathbb{DIST}\\
&  \subsetneq\mathbb{OML}\vee V_{\mathbb{BZL}}(\mathbb{AOL}).
\end{align*}
By Lemma \ref{sprodvar}, the above and Proposition \ref{d3saol},%
\[%
\begin{array}
[c]{ll}%
\mathbb{OML}\vee V_{\mathbb{BZL}}(\mathbf{D}_{5}) & =\mathbb{OML}%
\vee(\mathbb{SAOL}\cap\mathbb{DIST})\\
& =\mathbb{OML}\vee\{\mathbf{A}\in V_{\mathbb{BZL}}(\mathbb{AOL}%
):\mathbf{A}\vDash\{SDM,DIST\}\}\\
& =\{\mathbf{A}\in\mathbb{OML}\vee V_{\mathbb{BZL}}(\mathbb{AOL}%
):\mathbf{A}\vDash\{SDM,DIST\}\}\\
&
\begin{array}
[c]{l}%
=\{\mathbf{A}\in\mathbb{OML}\vee V_{\mathbb{BZL}}(\mathbb{AOL}):\mathbf{A}%
\vDash SDM\}\\
\cap\{\mathbf{A}\in\mathbb{OML}\vee V_{\mathbb{BZL}}(\mathbb{AOL}%
):\mathbf{A}\vDash DIST\}
\end{array}
\\
& =(\mathbb{OML}\vee\mathbb{SAOL})\cap(\mathbb{OML}\vee\mathbb{DIST}),
\end{array}
\]
hence $(\mathbb{OML}\vee\mathbb{SAOL})\cap(\mathbb{OML}\vee\mathbb{DIST}%
)\subsetneq\mathbb{OML}\vee\mathbb{SAOL},\mathbb{OML}\vee\mathbb{DIST}$, so
that $\mathbb{OML}\vee\mathbb{SAOL}$ and $\mathbb{OML}\vee\mathbb{DIST}$ are
incomparable. Therefore%
\begin{align*}
\mathbb{OML}\vee\mathbb{SAOL},\mathbb{OML}\vee\mathbb{DIST}  &  \subsetneq
\mathbb{OML}\vee\mathbb{SAOL}\vee\mathbb{OML}\vee\mathbb{DIST}\\
&  =\mathbb{OML}\vee\mathbb{SAOL}\vee\mathbb{DIST}.
\end{align*}
Since $\mathbb{OML}\vDash DIST\vee^{\sim}$ and $\mathbb{SAOL}\vee
\mathbb{DIST}\vDash WDSDM$, it follows that $\mathbb{OML}\vee\mathbb{SAOL}%
\vee\mathbb{DIST}\vDash WDIST\vee^{\sim}$. Note that, in $\mathbb{AOL}$,
$WDIST\vee^{\sim}$ is equivalent to $WDSDM$, hence, by the proof of
Proposition \ref{distnsaol}, the antiortholattice $\mathbf{M}_{3}%
\oplus\mathbf{M}_{3}$ fails $WDIST\vee^{\sim}$. It follows that $\mathbf{M}%
_{3}\oplus\mathbf{M}_{3}\in\mathbb{AOL}\setminus(\mathbb{OML}\vee
\mathbb{SAOL}\vee\mathbb{DIST})\subseteq(\mathbb{OML}\vee V_{\mathbb{BZL}%
}(\mathbb{AOL}))\setminus(\mathbb{OML}\vee\mathbb{SAOL}\vee\mathbb{DIST})$,
therefore $\mathbb{OML}\vee\mathbb{SAOL}\vee\mathbb{DIST}\subsetneq
\mathbb{OML}\vee V_{\mathbb{BZL}}(\mathbb{AOL})$.
\end{proof}

\begin{lemma}
\label{sdirirred}For any subvariety ${\mathbb{V}}$ of $\mathbb{OML}\vee
V_{\mathbb{BZL} }(\mathbb{AOL})$, $Si({\mathbb{V}})={\mathbb{V}}\cap
Si(\mathbb{OML}\cup\mathbb{AOL})$.
\end{lemma}

\begin{proof}
By Lemma \ref{subdirirred}.
\end{proof}

Note that, if a PBZ$^{\ast}$ --lattice $\mathbf{L}$ satisfies the SDM, then
$0$ is meet--irreducible in the join--subsemilattice $T(\mathbf{L})$ of
$\mathbf{L}$, but the converse does not hold.

\begin{lemma}
\label{sdmout} Let $\mathbf{A}$ be an antiortholattice without SDM and
$(\mathbf{A}_{i})_{i\in I}$ be a non--empty family of antiortholattices. Then:

\begin{itemize}
\item if $\mathbf{A}\in{\mathrm{S}}_{\mathbb{BZL} }(\prod_{i\in I}%
\mathbf{A}_{i})$, then the family $(\mathbf{A}_{i})_{i\in I}$ contains no
nontrivial antiortholattice with SDM;

\item $\mathbf{A}\in{\mathrm{S}}_{\mathbb{BZL} }(\prod_{i\in I}\mathbf{A}%
_{i})$ iff $\mathbf{A}\in{\mathrm{S}}_{\mathbb{BZL} }(\prod_{i\in
I,\mathbf{A}_{i}\nvDash SDM}\mathbf{A}_{i})$.
\end{itemize}
\end{lemma}

\begin{proof}
The second statement obviously follows from the first. Now assume that
$\mathbf{A}\in{\mathrm{S}}_{\mathbb{BZL} }(\prod_{i\in I}\mathbf{A}_{i})$, let
$J=\{j\in I:\mathbf{A}_{j}\vDash SDM\}$ and assume ex absurdo that there
exists a $k\in J$ such that $\mathbf{A}_{k}$ is nontrivial. We may consider
$A\subseteq\prod_{i\in I}A_{i}$. $\mathbf{A}$ is an antiortholattice that
fails SDM, in particular a nontrivial antiortholattice, hence there exist
$a=(a_{i})_{i\in I},b=(b_{i})_{i\in I}\in A\setminus\{0\}=D(\mathbf{A}%
)=D(\prod_{i\in I}\mathbf{A}_{i})=\prod_{i\in I}D(\mathbf{A}_{i})=\prod_{i\in
I}((A_{i}\setminus\{0\})\cup\{1\})$ such that $a\wedge b=0$, so that
$a_{k}\wedge b_{k}=0$ and $a_{k},b_{k}\in D(\mathbf{A}_{k})=A_{k}%
\setminus\{0\}$, which contradicts the fact that $\mathbf{A}_{k}$ satisfies
the SDM.
\end{proof}

\begin{proposition}
If ${\mathbb{V}} $ is a subvariety of $V_{\mathbb{BZL} }(\mathbb{AOL} )$,
then: ${\mathbb{V}} \vee\mathbb{SAOL} ={\mathbb{V}} \times_{s}\mathbb{SAOL} $
iff $({\mathbb{V}} \vee\mathbb{SAOL} )\cap\mathbb{AOL} =({\mathbb{V}}
\cup\mathbb{SAOL} )\cap\mathbb{AOL} $.
\end{proposition}

\begin{proof}
By the above, ${\mathbb{V}}\vee\mathbb{SAOL}={\mathbb{V}}\times_{s}%
\mathbb{SAOL}$ iff $Si({\mathbb{V}}\vee\mathbb{SAOL})=Si({\mathbb{V}}%
\cup\mathbb{SAOL})$. Since $Si(V_{\mathbb{BZL} }(\mathbb{AOL}))\subset
\mathbb{AOL}$, the right-to-left implication holds. Now assume that
$Si({\mathbb{V}}\vee\mathbb{SAOL})=Si({\mathbb{V}}\cup\mathbb{SAOL})$, and
assume ex absurdo that there exists an $\mathbf{L}\in(({\mathbb{V}}%
\vee\mathbb{SAOL})\cap\mathbb{AOL})\setminus({\mathbb{V}}\cup\mathbb{SAOL})$.
Then $\mathbf{L}\in{\mathbb{V}}\times_{s}\mathbb{SAOL}$, hence $\mathbf{L}%
\in{\mathrm{S}}_{\mathbb{BZL} }(\mathbf{A}\times\mathbf{B})$ for some
$\mathbf{A}\in{\mathbb{V}}$ and some $\mathbf{B}\in\mathbb{SAOL}$, therefore
$\mathbf{L}\in{\mathrm{S}}_{\mathbb{BZL} }(\mathbf{A}\times\prod_{j\in
J}\mathbf{B}_{j})$ for some family $(\mathbf{B}_{j})_{j\in J}\subseteq
\mathbb{SAOL}\cap\mathbb{AOL}$. Thus $\mathbf{L}\in{\mathrm{S}}_{\mathbb{BZL}
}(\mathbf{A})$ by Lemma \ref{sdmout}, so that $\mathbf{L}\in{\mathbb{V}}$, a
contradiction. Hence $({\mathbb{V}}\vee\mathbb{SAOL})\cap\mathbb{AOL}%
\subseteq({\mathbb{V}}\cup\mathbb{SAOL})\cap\mathbb{AOL}$.
\end{proof}

\section{Comparison with Other Structures\textbf{\label{thedis}}}

\subsection{Distributive Lattices with Two Unary Operations}

Bounded distributive lattices expanded both by a De Morgan complementation and
a unary operation with Stone-like properties have been the object of rather
intensive investigations over the past decades. In particular, Blyth, Fang and
Wang \cite{BFW} have studied, under the label of \emph{quasi-Stone De Morgan
algebras}, bounded distributive lattices with two unary operations that make
their appropriate reducts, at the same time, De Morgan algebras and
\emph{quasi-Stone algebras} \cite{San, Gai, Cel}. Quasi-Stone De Morgan
algebras that are simultaneously \emph{Stone} algebras and \emph{Kleene}
algebras are known under the name of \emph{Kleene-Stone algebras}; they have
been studied in \cite{GS} and, more recently, in the already quoted
\cite{BFW}. We begin this section by showing that the variety of
antiortholattices generated by the algebra $\mathbf{D}_{5}$ coincides with the
variety of Kleene-Stone algebras. This fact explains the similarity of some
results independently obtained in \cite{BFW, PBZ2, pbz5}.

\begin{definition}
A \emph{quasi-Stone algebra} is an algebra $\mathbf{A}=\left(  A,\wedge
,\vee,^{\sim},0,1\right)  $ of type $\left(  2,2,1,0,0\right)  $ such that
$\left(  A,\wedge,\vee,0,1\right)  $ is a bounded distributive lattice and the
unary operation $^{\sim}$ satisfies the following conditions for all $a,b\in
A$:%
\[%
\begin{tabular}
[c]{clcl}%
\textbf{QS1} & $0^{\sim}=1$ and $1^{\sim}=0$; & \textbf{QS4} & $a\leq
a^{\sim\sim}$;\\
\textbf{QS2} & $\left(  a\vee b\right)  ^{\sim}=a^{\sim}\wedge b^{\sim}$; &
\textbf{QS5} & $a^{\sim}\vee a^{\sim\sim}=1$.\\
\textbf{QS3} & $\left(  a\wedge b^{\sim}\right)  ^{\sim}=a^{\sim}\vee
b^{\sim\sim}$; &  &
\end{tabular}
\
\]
A quasi-Stone algebra $\mathbf{A}$ is a \emph{Stone algebra} if it
additionally satisfies SDM.
\end{definition}

The following useful lemma contains results to be found in \cite{San} and
\cite{BFW}:

\begin{lemma}
Let $\mathbf{A}=\left(  A,\wedge,\vee,^{\sim},0,1\right)  $ be a quasi-Stone
algebra. Then:

\begin{enumerate}
\item $\mathbf{A}$ satisfies the following conditions for all $a,b\in A$:%
\[%
\begin{tabular}
[c]{clcl}%
\textbf{QS6} & if $a\leq b$, then $b^{\sim}\leq a^{\sim}$; & \textbf{QS8} &
$a^{\sim\sim}{}^{\sim}=a^{\sim}$;\\
\textbf{QS7} & $a\wedge a^{\sim}=0$; & \textbf{QS9} & $a\wedge b^{\sim}=0$ iff
$a\leq b^{\sim\sim}$.
\end{tabular}
\]

\item The set $B\left(  \mathbf{A}\right)  =\left\{  a^{\sim}:a\in A\right\}
=\left\{  a\in A:a=a^{\sim\sim}\right\}  $ is a Boolean subuniverse of
$\mathbf{A}$.
\end{enumerate}
\end{lemma}

Clearly, in case $\mathbf{A}$ is a Stone algebra, the condition QS9 can be
strengthened to the pseudocomplementation equivalence:%
\[%
\begin{tabular}
[c]{cl}%
\textbf{S1} & $a\wedge b=0$ iff $a\leq b^{\sim}$ for all $a,b\in A$.
\end{tabular}
\
\]

\begin{definition}
A \emph{quasi-Stone De Morgan algebra} is an algebra $\mathbf{A}%
=(A,\wedge,\vee,^{\prime},$\linebreak$^{\sim},0,1)$ of type $\left(
2,2,1,1,0,0\right)  $ such that $\left(  A,\wedge,\vee,^{\prime},0,1\right)  $
is a De Morgan algebra, $\left(  A,\wedge,\vee,^{\sim},0,1\right)  $ is a
quasi-Stone algebra, and $a^{\prime}\in B\left(  \mathbf{A}\right)  $ whenever
$a\in B\left(  \mathbf{A}\right)  $. If $\left(  A,\wedge,\vee,^{\prime
},0,1\right)  $ is a Kleene algebra and $\left(  A,\wedge,\vee,^{\sim
},0,1\right)  $ is a (quasi-)Stone algebra, then $\mathbf{A}$ is said to be a
\emph{Kleene-(quasi-)Stone algebra}.
\end{definition}

\begin{lemma}
\label{mola}\cite{BFW} If $\mathbf{A}$ is a quasi-Stone De Morgan algebra,
then for all $a\in A$ we have that $a^{\sim\sim}=a^{\sim\prime\sim\prime}$.
\end{lemma}

Recall from Proposition \ref{distnsaol} that the variety generated by the
$5$-element antiortholattice chain $\mathbf{D}_{5}$ is axiomatised relative to
$\mathbb{PBZL}^{\mathbb{\ast}}$ by the lattice distribution axiom DIST and the
Strong De Morgan law SDM (J0 easily follows from these assumptions in the
context of $\mathbb{PBZL}^{\mathbb{\ast}}$). We now show that:

\begin{theorem}
$V_{\mathbb{BZL} }\left(  \mathbf{D}_{5}\right)  $ coincides with the variety
of Kleene-Stone algebras.
\end{theorem}

\begin{proof}
It is readily seen that $\mathbf{D}_{5}$ satisfies all the defining conditions
of Kleene-Stone algebras. Conversely, by the above remark, it will be
sufficient to show that Kleene-Stone algebras satisfy all the axioms of
PBZ$^{\ast}$ --lattices, since they are clearly distributive as lattices and
satisfy SDM by definition. We confine ourselves to the sole nontrivial items.
(i) The condition $\left(  \ast\right)  $, $\left(  x\wedge x^{\prime}\right)
^{\sim}=x^{\sim}\vee x^{\prime\sim}$ directly follows from SDM. (ii) We show
that $a^{\sim\sim}=a^{\sim\prime}$. By QS5, $a^{\sim}\vee a^{\sim\sim}=1$,
whence $a^{\sim\prime}\wedge a^{\sim\sim\prime}=0$. By S1, $a^{\sim\sim\prime
}\leq a^{\sim\prime\sim}$, whence, given the fact that $a^{\sim\sim}\in
B\left(  \mathbf{A}\right)  $,%
\[
a^{\sim\prime}\leq_{\left(  QS4\right)  }a^{\sim\prime\sim\sim}\leq_{\left(
QS6\right)  }a^{\sim\sim\prime\sim}=a^{\sim\sim}\text{.}%
\]
From this inequality, QS6 and QS8 we obtain that $a^{\sim}=a^{\sim\sim}%
{}^{\sim}\leq a^{\sim\prime\sim}$ and thus, by Lemma \ref{mola}, $a^{\sim\sim
}=a^{\sim\prime\sim\prime}\leq a^{\sim\prime}$. The converse inequality
follows from S1 and the fact that $a^{\sim}\in B\left(  \mathbf{A}\right)  $.
(iii) To round up our proof, it will suffice to show that any Kleene algebra
is paraorthomodular. Thus, let $a\leq b$ and $a^{\prime}\wedge b=0$. Then
$a^{\prime}\wedge a\leq a^{\prime}\wedge b=0$, whence $a$ is sharp and thus
$a\vee a^{\prime}=1$. As $a\wedge b=a$ and $a^{\prime}\wedge b=0$,
distributivity implies that%
\[
a=\left(  a\wedge b\right)  \vee\left(  a^{\prime}\wedge b\right)  =\left(
a\vee a^{\prime}\right)  \wedge b=1\wedge b=b.
\]

\end{proof}

The question as to whether the distributive subvariety $\mathbb{DIST}$ of
$V_{\mathbb{BZL}}\left(  \mathbb{AOL}\right)  $ coincides with the variety of
Kleene-quasi-Stone algebras is of a certain interest. The next Example answers
this problem in the negative.

\begin{example}
The BZ-lattice $\mathbf{BZ}_{4}$ (see \cite[Figure 5]{GLP1+}) is a
Kleene-quasi-Stone algebra, yet it is not even a member of $\mathbb{PBZL}%
^{\mathbb{\ast}}$. In fact, call $a$ and $a^{\prime}$ its two atoms. We have
that:%
\[
\left(  a\wedge a^{\prime}\right)  ^{\sim}=0^{\sim}=1\neq0=a^{\sim}\vee
a^{\prime\sim}.
\]

\end{example}

Finally, we prove that the variety generated by the $3$-element
antiortholattice chain $\mathbf{D}_{3}$ is a discriminator variety
\cite{Werner}.

\begin{proposition}
$V_{\mathbb{BZL} }\left(  \mathbf{D}_{3}\right)  $ is a discriminator variety.
\end{proposition}

\begin{proof}
Clearly, it suffices to find a ternary term that realises the discriminator
function on $\mathbf{D}_{3}$. Let first%
\[
e\left(  x,y\right)  =\left(  x^{\sim}\wedge\Diamond y\right)  \vee\left(
y^{\sim}\wedge\Diamond x\right)  \vee\left(  \square x\wedge\left(  \square
y\right)  ^{\sim}\right)  \vee\left(  \square y\wedge\left(  \square x\right)
^{\sim}\right)  \text{.}%
\]
It is a routinary matter to check that for all $a,b\in D_{3}$, $e^{\mathbf{D}%
_{3}}\left(  a,a\right)  =0$ and $e^{\mathbf{D}_{3}}\left(  a,b\right)  =1$ if
$a\neq b$. It follows that%
\[
t\left(  x,y,z\right)  =\left(  e\left(  x,y\right)  \vee z\right)
\wedge\left(  e\left(  x,y\right)  ^{\prime}\vee x\right)
\]
realises the discriminator function on $\mathbf{D}_{3}$.
\end{proof}

Observe that the algebra $\mathbf{D}_{3}$ fails to be primal, because it has
the nontrivial proper subuniverse $\left\{  0,1\right\}  $. Nonetheless, upon
identifying $D_{3}$ with the set of rational numbers $\left\{  0,\frac{1}%
{2},1\right\}  $, the truncated sum operation is definable as follows:%
\[
x\oplus y=\min\left(  1,x+y\right)  =\left(  x\vee\Diamond y\right)
\wedge\left(  y\vee\Diamond x\right)  \text{.}%
\]
It is easy to check that, upon expanding its signature by this binary
operation, $\mathbf{D}_{3}$ becomes an instance of a \emph{De Morgan
Brouwer-Zadeh MV-algebra} \cite{Cat1, Cat2} and, therefore, generates a
subvariety of such. The interest of this remark lies in the fact that the
variety of De Morgan Brouwer-Zadeh MV-algebras is known to be term-equivalent
to other well-known varieties of algebras of logic, including Heyting-Wajsberg
algebras, Stonean MV-algebras and MV algebras with Baaz Delta \cite{Cat3}. In
the next section, we will see that $V_{\mathbb{BZL}}\left(  \mathbf{D}%
_{3}\right)  $ is term-equivalent to another well-known variety of algebras of logic.

\subsection{Modal Algebras}

The standard examples of modal algebras (monadic algebras or interior
algebras, to name a few examples) were devised as the algebraic counterparts
of normal modal logics, which are extensions of classical propositional logic
--- therefore, they all have a Boolean algebra reduct. There is a thriving
literature, however, on \textquotedblleft nonstandard\textquotedblright\ modal
algebras based on generic De Morgan algebras: see below for the appropriate
references. The aim of this section is to chart this area of research and
locate term-equivalent counterparts of some distributive subvarieties of
PBZ$^{\ast}$ --lattices on this map. We consider algebras $\mathbf{M}=\left(
M,\wedge,\vee,^{\prime},\Diamond,0,1\right)  $ of type $\left(
2,2,1,1,0,0\right)  $, where $\left(  M,\wedge,\vee,^{\prime},0,1\right)  $ is
a De Morgan algebra. We assume that $^{\prime}$ binds stronger than $\Diamond
$, to reduce the number of parentheses. The following list of identities will
be crucial for defining the varieties that follow; henceforth, $\square x$ is
short for $\left(  \Diamond x^{\prime}\right)  ^{\prime}$.

\begin{description}
\item[M1] $\Diamond0\approx0$

\item[M2] $\Diamond\left(  x\vee y\right)  \approx\Diamond x\vee\Diamond y$

\item[M3] $x\leq\Diamond x$

\item[M4] $\Diamond x\approx\Diamond\Diamond x$

\item[M5] $\Diamond x\wedge\left(  \Diamond x\right)  ^{\prime}\approx0$

\item[M6] $\Diamond x\approx\square\Diamond x$

\item[M7] $\Diamond\left(  x\wedge x^{\prime}\right)  \approx\Diamond
x\wedge\Diamond x^{\prime}$

\item[M8] $x^{\prime}\vee\Diamond x\approx1$

\item[M9] $\Diamond\left(  x\wedge y\right)  \approx\Diamond x\wedge\Diamond
y$

\item[M10] $x\wedge x^{\prime}\approx\Diamond x\wedge x^{\prime}$
\end{description}

\begin{definition}
\label{firstbunch}

\begin{enumerate}
\item A $\Diamond$\emph{-De Morgan algebra} is an algebra $\mathbf{M}%
=(M,\wedge,\vee,^{\prime},\Diamond,\linebreak0,1)$ of type $\left(
2,2,1,1,0,0\right)  $, where $\left(  M,\wedge,\vee,^{\prime},0,1\right)  $ is
a De Morgan algebra and the identities M1 and M2 are satisfied.

\item A \emph{topological quasi-Boolean algebra} is a $\Diamond$-De Morgan
algebra satisfying the identities M3 and M4.

\item A \emph{classical }$\Diamond$\emph{-De Morgan algebra} is a topological
quasi-Boolean algebra satisfying the identity M5.

\item A \emph{monadic De Morgan algebra} is a classical $\Diamond$-De Morgan
algebra satisfying the identity M6.
\end{enumerate}
\end{definition}

$\Diamond$-De Morgan algebras and classical $\Diamond$-De Morgan algebras were
introduced in dual form by Sergio Celani \cite[pp. 253-254]{Cel}. Topological
quasi-Boolean algebras were first investigated by Banerjee and Chakraborty in
the context of the theory of rough sets \cite{Ban}. The authors of \cite{Saha}
also introduce, under the label of \emph{topological quasi-Boolean algebras
5}, a subvariety of topological quasi-Boolean algebras that satisfy M6 but not
M5. Clearly, topological quasi-Boolean algebras are meant to be a nonclassical
counterpart of interior algebras, while monadic De Morgan algebras can be
viewed as a nonclassical counterpart of monadic algebras. Condition M5, which
is of course trivial once our algebras have a Boolean nonmodal reduct, is
there to restore the Boolean behaviour of the nonmodal operators, when applied
to arguments of the form $\Diamond x$. Observe that all classical $\Diamond
$-De Morgan algebras satisfy the identity M8 \cite[Lemma 2.3]{Cel}.

There are several ways to strengthen the defining conditions of classical
$\Diamond$-De Morgan algebras with an eye to obtaining varieties with more
interesting properties.

\begin{enumerate}
\item A possible avenue is to impose on the possibility operator properties
that would determine a collapse of modality when the underlying structures are
Boolean algebras. For example,\emph{ tetravalent modal algebras} \cite{Mont,
Lour} are classical $\Diamond$-De Morgan algebras that satisfy M10, although
they are usually presented in a streamlined axiomatisation containing only the
axioms for De Morgan algebras plus M8 and M10. They form a discriminator
variety, generated by a quasiprimal four-element algebra (see item (iv) of the
proof of Theorem \ref{menarini} below).

\item On the other hand, one can enforce what Cattaneo et al. \cite{Cat4} call
a \textquotedblleft deviant\textquotedblright\ behaviour of the possibility
operator, requesting that it distribute not only over joins, but over meets as
well. \emph{Involutive Stone algebras} (\cite{Ciggal}; cp. also \cite{Cat4},
where these structures are called \emph{MDS5-algebras}), thus, are classical
$\Diamond$-De Morgan algebras satisfying M9. It is known that both involutive
Stone algebras and tetravalent modal algebras are monadic De Morgan algebras:
see \cite{Ciggal} and \cite[Proposition 1.2]{FR}, respectively.
\end{enumerate}

We now introduce the modal analogue of distributive PBZ$^{\ast}$ --lattices.

\begin{definition}
A \emph{weak \L ukasiewicz algebra} is a classical $\Diamond$-De Morgan
algebra $\mathbf{M}=\left(  M,\wedge,\vee,^{\prime},\Diamond,0,1\right)  $
such that its $\Diamond$-free reduct is a Kleene algebra and the identity M7
is satisfied.
\end{definition}

\begin{theorem}
\begin{enumerate}
\item Every weak \L ukasiewicz algebra $\mathbf{M}$ is a monadic De Morgan algebra.

\item The variety of weak \L ukasiewicz algebras is term-equivalent to
$\mathbb{DIST}$.
\end{enumerate}
\end{theorem}

\begin{proof}
(i) Let $a\in M$. Using M1, M5, M7 and M4, we have that%
\[
0=\Diamond0=\Diamond\left(  \Diamond a\wedge\left(  \Diamond a\right)
^{\prime}\right)  =\Diamond\Diamond a\wedge\Diamond\left(  \left(  \Diamond
a\right)  ^{\prime}\right)  =\Diamond a\wedge\Diamond\left(  \left(  \Diamond
a\right)  ^{\prime}\right)  \text{.}%
\]
Thus $\left(  \Diamond a\right)  ^{\prime}\vee\square\Diamond a=1$, whence, by
M5,%
\[
\Diamond a=\Diamond a\wedge\left(  \left(  \Diamond a\right)  ^{\prime}%
\vee\square\Diamond a\right)  =\left(  \Diamond a\wedge\left(  \Diamond
a\right)  ^{\prime}\right)  \vee\left(  \Diamond a\wedge\square\Diamond
a\right)  =\Diamond a\wedge\square\Diamond a\text{.}%
\]
Consequently, $\Diamond a\leq\square\Diamond a$. The converse inequality
follows from M3.

(ii) Let $\mathbf{M}=\left(  M,\wedge,\vee,^{\prime},\Diamond^{\mathbf{M}%
},0,1\right)  $ be a weak \L ukasiewicz algebra. We define $f\left(
\mathbf{M}\right)  $ as the algebra $\left(  M,\wedge,\vee,^{\prime},^{\sim
f\left(  \mathbf{M}\right)  },0,1\right)  $, where for all $a\in M$, $a^{\sim
f\left(  \mathbf{M}\right)  }=\left(  \Diamond^{\mathbf{M}}a\right)  ^{\prime
}$. Conversely, given a distributive PBZ$^{\ast}$ --lattice $\mathbf{L}%
=\left(  L,\wedge,\vee,^{\prime},^{\sim\mathbf{L}},0,1\right)  $, we define
$g\left(  \mathbf{L}\right)  $ as the algebra $\left(  L,\wedge,\vee,^{\prime
},\Diamond^{g\left(  \mathbf{L}\right)  },0,1\right)  $, where for all $a\in
L$, $\Diamond^{g\left(  \mathbf{L}\right)  }a=a^{\sim\mathbf{L}\sim\mathbf{L}%
}$. Clearly, $f\left(  \mathbf{M}\right)  $ has a Kleene lattice reduct. If
$a\in M$, then $a\wedge a^{\sim f\left(  \mathbf{M}\right)  }=a\wedge\left(
\Diamond^{\mathbf{M}}a\right)  ^{\prime}\leq\Diamond^{\mathbf{M}}%
a\wedge\left(  \Diamond^{\mathbf{M}}a\right)  ^{\prime}=0$, by M3 and M5.
Moreover,%
\[
a^{\sim f\left(  \mathbf{M}\right)  \sim f\left(  \mathbf{M}\right)  }=\left(
\Diamond^{\mathbf{M}}\left(  \Diamond^{\mathbf{M}}a\right)  ^{\prime}\right)
^{\prime}=\Diamond^{\mathbf{M}}a\geq a\text{,}%
\]
by M3 and item (1). For the same reason, $a^{\sim f\left(  \mathbf{M}\right)
\prime}=\left(  \Diamond^{\mathbf{M}}a\right)  ^{\prime\prime}=\Diamond
^{\mathbf{M}}a=a^{\sim f\left(  \mathbf{M}\right)  \sim f\left(
\mathbf{M}\right)  }$. Finally, by M2, whenever $a\leq b$,%
\[
\Diamond^{\mathbf{M}}b=\Diamond^{\mathbf{M}}\left(  a\vee b\right)
=\Diamond^{\mathbf{M}}a\vee\Diamond^{\mathbf{M}}b\text{,}%
\]
i.e. $\Diamond^{\mathbf{M}}a\leq\Diamond^{\mathbf{M}}b$, whence $b^{\sim
f\left(  \mathbf{M}\right)  }=\left(  \Diamond^{\mathbf{M}}b\right)  ^{\prime
}\leq\left(  \Diamond^{\mathbf{M}}a\right)  ^{\prime}\leq a^{\sim f\left(
\mathbf{M}\right)  }$. In sum, $f\left(  \mathbf{M}\right)  $ is a
distributive BZ-lattice. Condition $(\ast)$ holds because of M7. Similarly, by
reverse-engineering $g\left(  \mathbf{L}\right)  $, it is not hard to show
that it is a weak \L ukasiewicz algebra. To round off the proof, observe that
for $a\in L$,%
\begin{align*}
a^{\sim f\left(  g\left(  \mathbf{L}\right)  \right)  }  &  =\left(
\Diamond^{g\left(  \mathbf{L}\right)  }a\right)  ^{\prime}=a^{\sim
\mathbf{L}\sim\mathbf{L}\prime}=a^{\sim\mathbf{L}\sim\mathbf{L}\sim\mathbf{L}%
}=a^{\sim\mathbf{L}}\text{,}\\
\Diamond^{g\left(  f\left(  \mathbf{M}\right)  \right)  }a  &  =a^{\sim
f\left(  \mathbf{M}\right)  \sim f\left(  \mathbf{M}\right)  }=\left(
\Diamond^{\mathbf{M}}\left(  \Diamond^{\mathbf{M}}a\right)  ^{\prime}\right)
^{\prime}=\Diamond^{\mathbf{M}}a\text{.}%
\end{align*}
Thus, $f$ and $g$ are mutually inverse functions.
\end{proof}

Similar term-equivalence results with subvarieties of $\mathbb{PBZL}%
^{\mathbb{\ast}}$ are obtained in \cite{Cat1} and \cite{CN} for two special
subvarieties of weak \L ukasiewicz algebras.

\begin{definition}
\begin{enumerate}
\item \cite[Definition 4.2]{Cat1} A \emph{\L ukasiewicz algebra} is a
weak\linebreak\L ukasiewicz algebra that satisfies the identity M9.

\item \cite{Mont2} A \emph{three-valued \L ukasiewicz algebra} is a
\L ukasiewicz algebra that satisfies the identity M10.
\end{enumerate}
\end{definition}

Clearly, \L ukasiewicz algebras are exactly the involutive Stone algebras
whose $\Diamond$-free reduct is a Kleene lattice. There is a burgeoning
literature on three-valued \L ukasiewicz algebras, see e.g. \cite{Abad, Mont2,
Mois}. Three-valued \L ukasiewicz algebras can be equivalently characterised
as tetravalent modal algebras satisfying M9, in which case, the Kleene
identity follows from the axioms. They are also called \emph{pre-rough
algebras} in the literature \cite{Saha}.

\begin{theorem}
\label{fattura}\cite[Theorems 4.3 and 5.7]{Cat1}

\begin{enumerate}
\item The variety of \L ukasiewicz algebras is term-equivalent to
$V_{\mathbb{BZL}}\left(  \mathbf{D}_{5}\right)  $.

\item The variety of three-valued \L ukasiewicz algebras is term-equivalent to\linebreak
$V_{\mathbb{BZL}}\left(  \mathbf{D}_{3}\right)  $.
\end{enumerate}
\end{theorem}

Taking into account the remarks at the end of last section, it is evident that
$V_{\mathbb{BZL} }\left(  \mathbf{D}_{5}\right)  $ and $V_{\mathbb{BZL}
}\left(  \mathbf{D}_{3}\right)  $ have repeatedly resurfaced in many different
incarnations, with different choices of primitives or with different
axiomatisations. We collect many of the observations made thus far in the
following result.

\begin{theorem}
\label{menarini} The strict inclusions and incomparabilities depicted in the following diagram all hold:

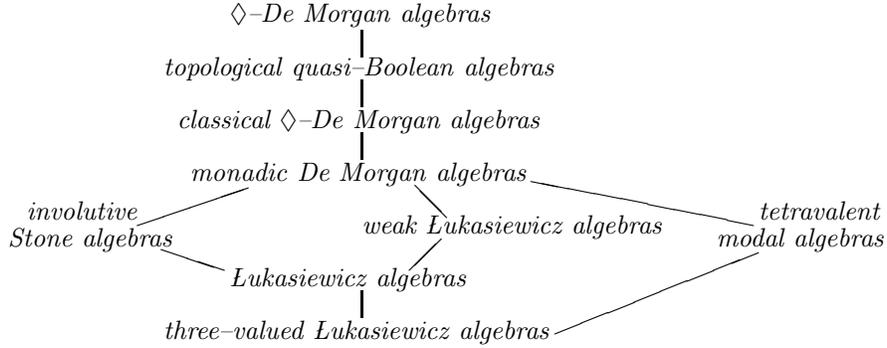
\begin{center}\hspace*{-15pt}\begin{picture}(250,125)(0,0)
\put(100,8){\line(0,1){10}}
\put(50,120){$\Diamond $--De Morgan algebras}
\put(100,107){\line(0,1){10}}
\put(25,100){topological quasi--Boolean algebras}
\put(100,88){\line(0,1){10}}
\put(30,80){classical $\Diamond $--De Morgan algebras}
\put(35,60){monadic De Morgan algebras}
\put(100,68){\line(0,1){10}}
\put(-27,45){involutive}
\put(-34,35){Stone algebras}
\put(15,43){\line(3,1){42}}
\put(24,34){\line(3,-1){24}}
\put(100,40){weak \L ukasiewicz algebras}
\put(250,45){tetravalent}
\put(234,35){modal algebras}
\put(248,43){\line(-5,1){84}}
\put(118,26){\line(1,1){12}}
\put(120,58){\line(1,-1){12}}
\put(50,20){\L ukasiewicz algebras}
\put(25,0){three--valued \L ukasiewicz algebras}
\put(173,2){\line(5,2){77}}
\end{picture}\end{center}\end{theorem}

\begin{proof} All that remains to be proved is that the inclusions are strict and that the varieties not connected by upward chains are incomparable.

\begin{enumerate}
\item Consider the algebra $\mathbf{D}_{2}$ as a De Morgan algebra, and let
$\Diamond0=\Diamond1=0$. This algebra is a $\Diamond$-De Morgan algebra which
is not a topological quasi-Boolean algebra.

\item Consider the algebra $\mathbf{D}_{3}$ as a De Morgan algebra, and let
$\Diamond x=x$ for all $x\in D_{3}=\left\{  0,a,1\right\}  $. This algebra is
a topological quasi-Boolean algebra which is not a classical $\Diamond$-De
Morgan algebra. In fact, $\Diamond a\wedge\left(  \Diamond a\right)  ^{\prime
}=a\neq0$.

\item Consider the algebra $\mathbf{D}_{2}^{2}$ as a De Morgan algebra with
universe $\left\{  0,a,a^{\prime},1\right\}  $, and let $\Diamond x=x$ for all
$x\in\left\{  0,a,1\right\}  $, and $\Diamond a^{\prime}=1$. This algebra is a
topological quasi-Boolean algebra which is not a monadic De Morgan algebra. In
fact, $\square\Diamond a=0\neq a=\Diamond a$.

\item Let $\mathbf{B}_{4}$ be the four-element algebra on $\left\{
0,a,b,1\right\}  $ that generates De Morgan algebras, with $a=a^{\prime}$ and
$b=b^{\prime}$. Let $\Diamond0=0$ and $\Diamond x=1$ for all $x\neq0$. This is
a tetravalent modal algebra (actually, it generates this variety), hence a
monadic De Morgan algebra, but not an involutive Stone algebra. In fact,
$\Diamond\left(  a\wedge b\right)  =0\neq1=\Diamond a\wedge\Diamond b$. Having
two fixpoints for the involution, it also fails to be a weak \L ukasiewicz
algebra, hence a \L ukasiewicz algebra or a three-valued \L ukasiewicz algebra.

\item Consider the algebra $\mathbf{D}_{2}^{2}$ as a De Morgan algebra with
universe $\left\{  0,a,a^{\prime},1\right\}  $, and let $\Diamond0=0$, and
$\Diamond x=1$ for all $x\neq0$. This algebra is a monadic De Morgan algebra
which is not a tetravalent modal algebra. In fact, $\Diamond a\wedge
a^{\prime}=a^{\prime}\neq0=a\wedge a^{\prime}$.

\item Consider the ordinal sum $\mathbf{D}_{2}\oplus$ $\mathbf{B}_{4}%
\oplus\mathbf{D}_{2}$ as a De Morgan algebra with universe $\left\{
0,a,b,c,a^{\prime},1\right\}  $, with $b=b^{\prime}$ and $c=c^{\prime}$, and
let $\Diamond0=0$, and $\Diamond x=1$ for all $x\neq0$. This algebra is an
involutive Stone algebra which is not a weak \L ukasiewicz algebra (or a
\L ukasiewicz algebra) since it has two fixpoints for the involution.

\item Consider the ordinal sum $\mathbf{D}_{2}^{2}\oplus\mathbf{D}_{2}^{2}$ as
a De Morgan algebra on $\{ 0,a,b,c,$\linebreak$b^{\prime},a^{\prime},1\} $,
with $c=c^{\prime}$, and let $\Diamond0=0$, and $\Diamond x=1$ for all
$x\neq0$. This is a weak \L ukasiewicz algebra which is not an involutive
Stone algebra, for $\Diamond\left(  a\wedge b\right)  =0\neq1=\Diamond
a\wedge\Diamond b$. A fortiori, it fails to be a \L ukasiewicz algebra.

\item Finally, consider the algebra $\mathbf{D}_{4}$ as a De Morgan algebra on
$\left\{  0,a,a^{\prime},1\right\}  $, and let $\Diamond0=0$, and $\Diamond
x=1$ for all $x\neq0$. This is a \L ukasiewicz algebra, hence both an
involutive Stone algebra and a weak \L ukasiewicz algebra. However, it fails
to be a tetravalent modal algebra (hence a three-valued \L ukasiewicz
algebra), for $\Diamond a\wedge a^{\prime}=a^{\prime}\neq a=a\wedge a^{\prime
}$.
\end{enumerate}
\end{proof}

\section*{Acknowledgements}

This work was supported by the research grants \textquotedblleft Propriet\`{a}
d`ordine nella semantica algebrica delle logiche non
classiche\textquotedblright, Regione Autonoma della Sardegna, L. R. 7/2007, n.
7, 2015, CUP: F72F16002920002, and \textquotedblleft Theory and applications
of resource sensitive logics\textquotedblright, PRIN 2017, Prot. 20173WKCM5,
CUP: F74I19000720001.

The authors thank Davide Fazio for the insightful discussions on the topics of the present paper.

\end{document}